\documentclass[11pt]{article}
\usepackage{amsmath,amsfonts,amssymb, amsthm,enumerate,graphicx}
\usepackage{tikz,authblk}
\usepackage{subfig}
\usepackage{url}
\usepackage{hyperref}
\usepackage{stackengine,scalerel}
\usetikzlibrary {decorations.pathmorphing, decorations.pathreplacing, decorations.shapes}

\usetikzlibrary{arrows.meta}

\newtheorem{theorem} {{\textsf{Theorem}}}
\newtheorem{proposition}[theorem]{{\textsf{Proposition}}}
\newtheorem{corollary}[theorem]{{\textsf{Corollary}}}

\newtheorem{remark}[theorem]{{\textsf{Remark}}}
\newtheorem{example}[theorem]{{\textsf{Example}}}
\newtheorem{lemma}[theorem]{{\textsf{Lemma}}}

\newcommand{\xdownarrow}[1]{
	{\left\downarrow\vbox to #1{}\right.\kern-\nulldelimiterspace}
}
\def\Z{{\mathbb{Z}^{n}_2}}
\begin{document}

\title{Crystallizations of small covers over the $n$-simplex $\Delta^n$ and the prism $\Delta^{n-1} \times I$}

\author{Anshu Agarwal$^1$ and Biplab Basak}
	
\date{February 10, 2026}
	
\maketitle
	
\vspace{-10mm}
\begin{center}
		
\noindent {\small Department of Mathematics, Indian Institute of Technology Delhi, New Delhi, 110016, India.$^2$}

\footnotetext[1]{Corresponding author}
		
\footnotetext[2]{{\em E-mail addresses:} \url{maz228084@maths.iitd.ac.in} (A. Agarwal), \url{biplab@iitd.ac.in} (B. Basak).}
		
\medskip
	
\end{center}
	
\hrule
	
\begin{abstract}
A crystallization of a PL manifold is an edge-colored graph that corresponds to a contracted triangulation of the manifold, facilitating the study of its topological and combinatorial properties. A small cover over a simple convex $n$-polytope $P^n$ is a closed $n$-manifold with a locally standard $\mathbb{Z}_2^n$-action such that its orbit space is homeomorphic to $P^n$. In this article, we study the crystallizations of small covers over the $n$-simplex $\Delta^n$ and the prism $\Delta^{n-1} \times I$. It is known that the small cover over the $n$-simplex $\Delta^n$ is $\mathbb{RP}^n$. For every $n\geq 2$, we prove that $\mathbb{RP}^n$ has a unique $2^n$-vertex crystallization. We also demonstrate that there are exactly $1 + 2^{n-1}$ D-J equivalence classes of small covers over the prism $\Delta^{n-1} \times I$, where $n\geq 3$. For each $\mathbb{Z}_2$-characteristic function of $\Delta^{n-1} \times I$, we construct a $2^{n-1}(n+1)$-vertex crystallization of the small cover $M^n(\lambda)$ with regular genus $1 + 2^{n-4}(n^2 - 2n - 3)$, where $n\geq 4$. The regular genus of closed PL \(n\)-manifolds extends the notions of the genus of surfaces and the Heegaard genus of 3-manifolds to higher dimensions. In this article, we construct four orientable and four non-orientable $\mathbb{RP}^3$-bundles over $\mathbb{S}^1$ up to D-J equivalence, each with regular genus $6$. Although the four orientable (resp. non-orientable) small covers are not D-J equivalent, we show that they are PL homeomorphic.
\end{abstract}

\noindent {\small {\em MSC 2020\,:} Primary 57Q15; Secondary 05C15, 57S25, 52B11, 52B70.
		
}
	
\medskip

\section{Introduction}
A crystallization of a PL manifold is a combinatorial representation using an edge-colored graph, providing an effective way to study its topological and combinatorial properties (cf. Subsection \ref{crystal}). This approach is particularly valuable as it allows for a discrete encoding of PL manifolds, facilitating computational and theoretical investigations. However, finding a crystallization for a given PL manifold is generally a challenging problem. In our article, we address this challenge by constructing crystallizations for a specific category of PL manifolds known as small covers. These manifolds arise from simple polytopes through characteristic functions, making them a rich class of spaces with combinatorial significance (cf. Subsection \ref{smallcover}). Leveraging this structure, we provide an explicit method to construct crystallizations for small covers, demonstrating how their combinatorial and topological properties can be studied effectively using this framework. Throughout this article, we work in the PL category unless stated otherwise, following the framework established in \cite{rs72}.

The concept of small covers, introduced by Davis and Januszkiewicz \cite{dj91}, has recently emerged as a captivating topic in toric topology. A closed manifold \( M^n \) is defined as a small cover if it supports a locally standard \(\mathbb{Z}_2^n\)-action, with its orbit space homeomorphic to a simple convex \( n \)-polytope \( P^n \). In the last three decades, extensive research has been conducted on small covers. Research on the topological types of 3-dimensional small covers is detailed in \cite{LY11}. Additionally, the enumeration of small covers over certain polytopes and the study of vector bundles over Davis-Januszkiewicz spaces are explored in \cite{C08, N10}. 
A small cover is a topological space closely linked to combinatorics through its relationship with simple polytopes. This space is constructed by assigning a characteristic function to the polytope's facets, encoding a \( \mathbb{Z}_2^n \)-action on the manifold. The study of small covers through crystallizations offers combinatorial techniques to simplify, visualize, and classify their topological properties, highlighting the rich interplay between combinatorial geometry and topology. A detailed study of crystallizations for 3-manifolds is available in \cite{BCG10, bd14,  CS08, KMN07}, while for 4-manifolds, it can be found in \cite{cc15}. 

In this article, we conduct a detailed study of the crystallizations of small covers over the \(n\)-simplex \(\Delta^n\) and the prism \(\Delta^{n-1} \times I\). It is a well-established fact that the small cover over the \(n\)-simplex \(\Delta^n\) is the real projective space \(\mathbb{RP}^n\) \cite{dj91}. It is also known that \(\mathbb{RP}^n\) admits a minimal crystallization consisting of \(2^n\) vertices. Firstly, we demonstrate that this crystallization is derived from our method of obtaining a colored graph from a small cover. We provide a proof that this \(2^n\)-vertex crystallization of \(\mathbb{RP}^n\) is unique, where $n\geq 2$ (cf. Theorem \ref{theorem:unique}).

Secondly, we investigate the small covers over the prism \(\Delta^{n-1} \times I\). We present a construction to obtain a gem of the small covers over the prism \(\Delta^{n-1} \times I\). We establish that there are precisely \(1 + 2^{n-1}\) Davis-Januszkiewicz (D-J) equivalence classes of these small covers, where $n\geq 3$ (cf. Lemma \ref{lemma:D-J}). For each \(\mathbb{Z}_2\)-characteristic function $\lambda$ of \(\Delta^{n-1} \times I\), we construct a \(2^{n-1}(n+1)\)-vertex crystallization of the corresponding small cover \(M^n(\lambda)\), where $n\geq 3$. We also compute the regular genus of these crystallizations, which is  \(1 + 2^{n-4}(n^2 - 2n - 3)\) for $n\ge 4$ (cf. Theorem \ref{main}).

In particular, we obtain crystallizations of small covers with regular genus 6 in the 4-dimensional case. The regular genus of closed \(n\)-manifolds extends the classical notions of the genus for surfaces and the Heegaard genus for 3-manifolds to higher dimensions. It serves as an important PL invariant. The classification of PL $n$-manifolds based on regular genus is a well-established problem in combinatorial topology. The topological classification of orientable PL $4$-manifolds up to regular genus $5$ is known (cf. \cite{cv99, cm93, s99}), but the classification remains open for regular genus 6 and higher. Some prime closed 4-manifolds with regular genus 6 can be found in  \cite{B19}. Here, we construct four orientable and four non-orientable \(\mathbb{RP}^3\)-bundles over \(\mathbb{S}^1\) up to D-J equivalence, each with regular genus 6 (cf. Corollary \ref{corollary:regular genus 6}). Although the four orientable (resp.\ non-orientable) small covers are not D-J equivalent, we show that they are PL homeomorphic (cf. Remark \ref{remark: PL isomorphic}).

Calculating the fundamental groups of small covers is an interesting problem (cf.~\cite{WY21}). Several algorithms exist to determine a presentation of the fundamental group of an $n$-manifold from its crystallizations (cf \cite{cc23, fgg86}). Since we obtain the crystallizations of small covers, presentations of their fundamental groups can be easily computed 
using these known algorithms.

\section{Preliminaries}\label{pre}

\subsection{Small Cover}\label{smallcover}
Davis and Januszkiewicz introduced the concept of small cover over a simple polytope in \cite{dj91}, which has since been studied extensively (see, for example, \cite{C08, LY11, N10, WY21}). As discussed in \cite[Remark, p. 421]{dj91}, throughout this article we consider all small covers and maps within the PL category. A {\it simple $n$-polytope} is a convex polytope such that exactly $n$ codimension-one faces meet at each vertex \cite{bp02}. For example, in platonic solids, a tetrahedron, cube, and dodecahedron are simple $3$-polytopes, while octahedron and icosahedron are not simple. Let $\rho$ be the standard action of $\Z$ on $\mathbb{R}^n$. A $\Z$ action $\eta$ on an $n$-dimensional manifold $M^n$ is called a {\it locally standard action} if for each $x\in M^n$, there exists an automorphism $\theta_x$ of $\Z$, a $\Z$-stable open neighborhood $U_x$ of $x$, and a $\Z$-stable open set $V_x$ in $\mathbb{R}^n$ such that $U_x$ and $V_x$ are $\theta_x$-equivariantly homeomorphic. That is, there is a homeomorphism $f_x:U_x\to V_x$ such that $$f_x(\eta(g,u))=\rho(\theta_x(g),f_x(u)).$$ Further, if the orbit space of this action $\eta$ is a simple convex $n$-polytope $P^n$, then we say that $M^n$ is a {\it small cover} over $P^n$. Therefore, we have a projection map $\pi: M^n \to P^n$ such that $\pi(x)$ is the orbit class of $x$ for all $x\in M^n$.

Given a simple $n$-polytope $P^n$, let $\mathcal F(P^n)$ denote the set of $(n-1)$-faces of $P^n$. A function $$\lambda:\mathcal{F}(P^n) \to \Z$$ is called a {\it $\mathbb Z_2$-characteristic function} if, for each vertex $v=\bigcap_{i=1}^{n} F_i,$ the vectors $\lambda(F_i),\ 1\le i \le n$, forms a basis of $\Z$, where $F_i\in \mathcal F (P^n)$. The vector $\lambda(F)$ is called the {\it $\mathbb{Z}_2$-characteristic vector} of $F,$ where $F\in \mathcal F (P^n).$ Let $G_F$ be the $l$-dimensional subspace generated by $\lambda(F_i),\ 1\le i\le l$, where $F=\bigcap_{i=1}^{l} F_i,\ F_i\in \mathcal F (P^n)$, a face of codimension-$l$. Define an equivalence relation on $\Z\times P^n$ as $$(g_1,p) \sim (g_2,p) \iff  \begin{cases}
g_1=g_2 & \text{if} \ p\in \text{int}(P^n)\\
g_1+g_2 \in G_{F_p} & \text{if} \ p\in \partial(P^n)
\end{cases},$$ where $F_p$ is the unique face containing $p$ in its relative interior. Let us denote the manifold $(\Z \times P^n)/ \sim $ by $M^n(\lambda)$. It is easy to check that the $\Z$-action $\eta$ on $M^n(\lambda)$ defined as $\eta(g,(g_1,p))=(g+g_1,p)$ is a locally standard action and its orbit space is $P^n$. Therefore, $M^n(\lambda)$ is a small cover over $P^n.$ 

 \begin{figure}[h!]
\tikzstyle{ver}=[]
\tikzstyle{edge} = [draw,thick,-]
    \centering
\begin{tikzpicture}[scale=0.6]
\foreach \x/\y/\z in
{-2.5/1.5/M_1^n,2.5/1.5/M_2^n,-2.5/-1.5/P^n,2.5/-1.5/P^n,0/2/f,0/-1/Id,-3/0/\pi_1,3/0/\pi_2}
{\node[ver] () at (\x,\y){$\z$};}
\draw [->](-1.8,1.5)--(1.8,1.5);
\draw [->](-2.5,1.1)--(-2.5,-1.1);
\draw [->](-1.8,-1.5)--(1.8,-1.5);
\draw [->](2.5,1.1)--(2.5,-1.1);
\end{tikzpicture}
\caption{Diagram depicting D-J equivalence of $M_1^n$ and $M_2^n$.}\label{fig:0}
\end{figure}

Let $M^n_1$ and $M^n_2$ be two small covers over $P^n$. The small covers $M^n_1$ and $M^n_2$ are called {\it D-J equivalent} if there exists a $\theta$-equivariant homeomorphism $f:M^n_1\to M^n_2$, covering the identity on $P^n$, where $\theta$ is an automorphism of $\Z$. In short, the diagram in Figure \ref{fig:0} commutes. It is evident that two small covers $M^n(\lambda_1)$ and $M^n(\lambda_2)$ are D-J equivalent if and only if there exists an automorphism $\theta$ of $\Z$ such that $\lambda_2=\theta \circ \lambda_1$. If $M^n$ is a small cover over $P^n$, then there exists a $\mathbb{Z}_2$-characteristic function $\lambda:\mathcal F(P^n)\to \Z$ such that $M^n(\lambda)$ and $M^n$ are equivariantly homeomorphic, covering the identity on $P^n$.

Let \( P^m \) and \( P^n \) be \( m \)- and \( n \)-polytopes, respectively, with \(\mathbb{Z}_2\)-characteristic functions \(\lambda_m: \mathcal{F}(P^m) \to \mathbb{Z}_2^m\) and \(\lambda_n: \mathcal{F}(P^n) \to \mathbb{Z}_2^n\). The set of \((m+n-1)\)-faces of \( P^m \times P^n \) is given by \(\mathcal{F}(P^m \times P^n) = \{F \times P^n, P^m \times F' \mid F \in \mathcal{F}(P^m), F' \in \mathcal{F}(P^n)\}\). Define a \(\mathbb{Z}_2\)-characteristic function \(\lambda: \mathcal{F}(P^m \times P^n) \to \mathbb{Z}_2^{m+n}\) by \(\lambda(F \times P^n) = (\lambda_m(F), \mathbf{0})\) and \(\lambda(P^m \times F') = (\mathbf{0}, \lambda_n(F'))\). 

\begin{proposition}[\cite{dj91}, 1.10] \label{Proposition: product}
Let \( P^m \) and \( P^n \) be \(m \)- and \( n \)-polytopes, respectively. Let $\lambda_m$ and $\lambda_n$ be the $\mathbb{Z}_2$-characteristic functions on $P^m$ and $P^n$, respectively. Let $\lambda$ be the $\mathbb{Z}_2$-characteristic function on $P^m \times P^n$ as defined above. Then the small cover $M^{m+n}(\lambda)$ over $P^m \times P^n$ is PL homeomorphic to $M^m(\lambda_m) \times M^n(\lambda_n)$.
\end{proposition}

\subsection{Crystallization} \label{crystal}
Suppose that $K$ is a finite collection of closed balls and write $|K| = \bigcup_{B\in K} B $. Then $K$ is called a simplicial cell complex if the following conditions hold.

\begin{enumerate}[$(i)$]
\item $|K|=$ $\bigsqcup_{B\in K}$ int$(B)$,

\item if $A,B\in K$, then $A\cap B$ is a union of balls of $K$,

\item  for each $h$-ball $A\in K$, the poset $\{B\in K \, | \, B \subset A\}$, ordered by inclusion, is isomorphic with the lattice of all faces of the standard $h$-simplex.
\end{enumerate}
\noindent A pseudo-triangulation of a polyhedron $P$ is a pair $(K,f)$, where $K$ is a simplicial cell complex and $f: |K| \to P$ is a homeomorphism (see \cite{fgg86} for more details). A maximal dimensional closed ball of $K$ is called a {\it facet}. If all the facets of $K$ are of the same dimension, then $K$ is called a {\it pure} simplicial cell complex. 

The crystallization theory provides a tool for representing piecewise-linear (PL) manifolds of any dimension combinatorially, using edge-colored graphs. Throughout the article, by a graph, we mean a multigraph with no loops. Let \(\Gamma = (V(\Gamma), E(\Gamma))\)  be an edge-colored graph, where the edges are colored (or labeled) using \(\Delta_n := \{0, 1, \dots, n\}\). The elements of the set \(\Delta_n\) are referred to as the {\it colors} of \(\Gamma\). The coloring of $\Gamma$ is called a \textit{proper edge-coloring} if any two adjacent edges in $\Gamma$ have different labels. In other words, for a proper edge-coloring, there exists a map $\gamma: E(\Gamma) \to \Delta_n$ such that  $\gamma(e_1) \ne \gamma(e_2)$ for any two adjacent edges $e_1$ and $e_2$. We denote a properly edge-colored graph as \((\Gamma,\gamma)\), or simply as \(\Gamma\) if the coloring is understood. If a graph $\Gamma$ is such that the degree of each vertex in the graph is $n+1$, then it is said to be {\it $(n+1)$-regular}.
We refer to \cite{bm08} for standard terminologies on graphs.

An {\it $(n+1)$-regular colored graph} is a pair $(\Gamma,\gamma)$, where $\Gamma$ is $(n+1)$-regular and $\gamma$ is a proper edge-coloring of $\Gamma$. For each $\mathcal{C} \subseteq \Delta_n$ with cardinality $k$, the graph $\Gamma_\mathcal{C} = (V(\Gamma), \gamma^{-1}(\mathcal{C}))$ is a $k$-regular colored graph with edge-coloring $\gamma|_{\gamma^{-1}(\mathcal{C})}$. For a color set $\{j_1,j_2,\dots,j_k\} \subset \Delta_n$, $g(\Gamma_{\{j_1,j_2, \dots, j_k\}})$ or $g_{\{j_1, j_2, \dots, j_k\}}$ denotes the number of connected components of the graph $\Gamma_{\{j_1, j_2, \dots, j_k\}}$. A graph $(\Gamma,\gamma)$ is called {\it contracted} if the subgraph $\Gamma_{\hat{j}} = \Gamma_{\Delta_n\setminus \{j\}}$ is connected, i.e., $g_{\hat{j}}=1$ for all $j \in \Delta_n$.
 
For a properly edge-colored graph $(\Gamma,\gamma)$ with the color set $\Delta_n$, a corresponding $n$-dimensional simplicial cell complex ${\mathcal K}(\Gamma)$ is constructed as follows:

\begin{itemize}
\item{} For each vertex $v\in V(\Gamma)$, take an $n$-simplex $\sigma(v)$ with vertices labeled by $\Delta_n$.

\item{} Corresponding to each edge of color $j$ between $v_1,v_2\in V(\Gamma)$, identify the ($n-1$)-faces of $\sigma(v_1)$ and $\sigma(v_2)$ opposite to the $j$-labeled vertices such that the vertices with the same labels coincide.
\end{itemize}

The simplicial cell complex \(\mathcal{K}(\Gamma)\) is \((n+1)\)-colorable, i.e., its 1-skeleton admits a proper vertex-coloring by \(\Delta_n\). Note that all colors of \(\Delta_n\) need not appear in the edge-coloring of \((\Gamma,\gamma)\). The topological space \( |\mathcal{K}(\Gamma)| \) inherits a natural PL structure, and the graph \((\Gamma,\gamma)\) is said to represent \( |\mathcal{K}(\Gamma)| \). In this article, we primarily consider \((\Gamma,\gamma)\) as an \((n+1)\)-regular colored graph. In that case, in \(\mathcal{K}(\Gamma)\) every \((n-1)\)-simplex is contained in exactly two \(n\)-simplices.
If \( |\mathcal{K}(\Gamma)| \) is homeomorphic to an \( n \)-manifold \( M \), then \((\Gamma, \gamma)\) is referred to as a \textit{gem (graph encoded manifold)} of \( M \), or, simply, it is said that \((\Gamma, \gamma)\) represents \( M \). In this context, \(\mathcal{K} (\Gamma)\) is described as a \textit{colored triangulation} of \( M \). 
The {\it disjoint star} of \(\sigma \in \mathcal{K}(\Gamma)\) is a simplicial cell complex that consists of all the \(n\)-simplices of \(\mathcal{K}(\Gamma)\) that contain \(\sigma\), with re-identification of only their \((n-1)\)-faces containing \(\sigma\) as in \(\mathcal{K}(\Gamma)\). The {\it disjoint link} of \(\sigma \in \mathcal{K}(\Gamma)\) is the subcomplex of its disjoint star generated by the simplices that do not intersect $\sigma$.

From the construction above, it can be easily seen that for any subset \(\mathcal{C} \subset \Delta_n\) with cardinality \(k+1\), \(\mathcal{K}(\Gamma)\) has as many \( k \)-simplices with vertices labeled by \(\mathcal{C}\) as there are connected components of \(\Gamma_{\Delta_n \setminus \mathcal{C}}\) \cite{fgg86}. Specifically, each component of the \((n-k)\)-regular colored subgraph induced by the colors from \(\Delta_n \setminus \mathcal{C}\) corresponds to the disjoint link of a \( k \)-simplex with vertices labeled by \(\mathcal{C}\). For further information on CW complexes and related concepts, refer to \cite{bj84}. An $(n+1)$-regular colored gem $(\Gamma,\gamma)$ of a closed manifold $M$ is called a {\em crystallization} of $M$ if it is contracted. In other words, the corresponding simplicial cell complex \( \mathcal{K}(\Gamma)\) has exactly $(n+1)$ vertices.

If \( K \) is a colored triangulation of an \( n \)-manifold \( M \), meaning that \( K \) is an \( (n+1) \)-colorable simplicial cell complex and \( |K| \) is homeomorphic to \( M \), then by reversing the steps of the above construction, we obtain a gem \( (\Gamma,\gamma) \) of \( M \). Clearly, \( \mathcal{K}(\Gamma)=K\). 
Every closed PL \(n\)-manifold \(M\) is known to admit a gem. From a gem, a crystallization of \(M\) can be easily obtained through certain combinatorial moves (see \cite{ fg82i, fgg86} for more details). Furthermore, it is well established in the literature that a gem of a closed manifold \(M\) is bipartite if and only if \(M\) is orientable.

Let $(\Gamma,\gamma)$ be an $(n+1)$-regular colored graph representing a closed manifold $M$. Let $\Lambda_1\subset V(\Gamma)$ and $\Lambda_2 \subset V(\Gamma)$ be such that the subgraphs $A_1$ and $A_2$ generated by $\Lambda_1$ and $\Lambda_2$, respectively, represent $n$-dimensional balls. Let there be an isomorphism $\Phi:A_1 \to A_2$ such that $u$ and $\Phi(u)$ are joined by an edge of color $i$ for each $u\in \Lambda_1$, and $\Lambda_1$ and  $\Lambda_2$ lie in different components of $\Gamma_{\hat{i}}$. Consider a new  $(n+1)$-colored graph  $\Gamma^\prime$ obtained from $\Gamma$ as follows. Let $V(\Gamma^\prime)=V(\Gamma)\setminus (\Lambda_1 \cup \Lambda_2)$. For two vertices $p$ and $q$ in $V(\Gamma^\prime)$, if $p$ and $q$ are connected to $u$ and $\Phi(u)$, respectively, by an edge of color $j\in \Delta_n \setminus  \{i\}$ in $\Gamma$ where $u\in \Lambda_1$, then $p$ and $q$ are joined by an edge of color $j$ in $\Gamma^\prime$. On the other hand, if $p$ and $q$ are joined by an edge of color $j\in \Delta_n$ in $\Gamma$, then $p$ and $q$ are joined by an edge of color $j$ in $\Gamma^\prime$. The process to obtain $\Gamma^\prime$ from $\Gamma$ is called a {\it polyhedral glue move} with respect to $(\Phi,\Lambda_1,\Lambda_2,i)$. From \cite{fg82i}, it is known that $\Gamma^\prime$ also represents $M$. If $\Lambda_1$ and $ \Lambda_2$ are singleton sets, then this polyhedral glue move is called a {\it simple glue move} or {\it cancellation of $1$-dipole}, where $\Lambda_1$ and $\Lambda_2$ forms $1$-dipole with respect to the color $i$. For more details, one can see \cite{fg82i}.

Let $(\Gamma,\gamma)$ and $(\bar{\Gamma},\bar{\gamma})$ be two  $(n+1)$-regular colored graphs with color sets $\Delta_n$ and $\bar{\Delta}_n$, respectively. Then $I:=(I_V,I_c):\Gamma \to \bar{\Gamma}$ is called an {\em isomorphism} if $I_V: V(\Gamma) \to V(\bar{\Gamma})$ and $I_c:\Delta_n \to \bar{\Delta}_n$ are bijective maps such that $uv$ is an edge of color $i \in \Delta_n$ if and only if $I_V(u)I_V(v)$ is an edge of color $I_c(i) \in \bar{\Delta}_n$. The graphs $(\Gamma, \gamma)$ and $(\bar{\Gamma}, \bar{\gamma})$ are then said to be {\it isomorphic}.

\subsection{Regular Genus of closed PL $n$-manifolds}\label{sec:genus}
The regular genus of closed \(n\)-manifolds extends the notions of the genus of surfaces and the Heegaard genus of 3-manifolds to higher dimensions. For a closed connected surface, its regular genus is simply its genus. However, for
closed connected $n$-manifolds ($n \geq 3$), the regular genus is defined as follows.
From \cite{fg82, g81}, it is known that if $(\Gamma,\gamma)$ is a bipartite (resp. non-bipartite) $(n+1)$-regular colored graph that represents a closed connected orientable (resp. non-orientable) $n$-manifold $M$, then for each cyclic permutation $\varepsilon=(\varepsilon_0,\dots,\varepsilon_n)$ of $\Delta_n$, there exists a regular embedding of $\Gamma$ into an orientable (resp. non-orientable) surface $S$. A {\it regular embedding} is an embedding where each region is bounded by a bi-colored cycle with colors $\varepsilon_i,\varepsilon_{i+1}$ for some $i$ (addition is modulo $n + 1$). Moreover, the Euler characteristic $\chi_\varepsilon(\Gamma)$ of the orientable (resp. non-orientable) surface  $S$ satisfies
$$\chi_\varepsilon(\Gamma)=\sum_{i \in \mathbb{Z}_{n+1}}g_{\varepsilon_i\varepsilon_{i+1}} + (1-n)\frac{\text{card}(V(\Gamma))}{2},$$ 
and the genus (resp. half the genus) $\rho_ \varepsilon$ of $S$ satisfies
$$\rho_ \varepsilon(\Gamma)=1-\frac{\chi_\varepsilon(\Gamma)}{2}.$$
The regular genus $\rho(\Gamma)$ of $(\Gamma,\gamma)$ is defined as
$$\rho(\Gamma)= \min \{\rho_{\varepsilon}(\Gamma) \ | \  \varepsilon \ \mbox{ is a cyclic permutation of } \ \Delta_n\}.$$
The regular genus of $M$ is defined as 
$$\mathcal G(M) = \min \{\rho(\Gamma) \ | \  (\Gamma,\gamma) \mbox{ represents } M\}.$$
	
The regular genus is a PL invariant. Some studies on the regular genus of 3-manifolds can be found in \cite{BCG10, CS08, KMN07}. A closed manifold of dimension \( n \) with regular genus \( 0 \) is characterized as \(\mathbb{S}^n\) \cite{fg82}. The following result gives a lower bound for the regular genus of a closed connected $4$-manifold.

\begin{proposition}[\cite{bc17}]\label{lbrg}
  Let $M$ be a closed connected PL $4$-manifold with $rk(\pi_1(M))=m$. Then $\mathcal G(M)\ge 2\chi(M)+5m-4$.
\end{proposition}

\section{Main Results}

\subsection{Uniqueness of $2^n$-vertex crystallization of $\mathbb{RP}^n$}\label{3.1}
Let $P^n=[v_0,v_1,\dots,v_n]$ be an $n$-simplex, where $n\ge 2$, and let $\mathcal F=\{F_i=[v_0,\dots,v_{i-1},v_{i+1},\dots,$ $v_n]\ |\ 0\le i\le n\}$ denote the set of $(n-1)$-faces of $P^n$. Let $\lambda:\mathcal F\to \Z$ be a $\mathbb{Z}_2$-characteristic function, and let $b_i$ denote the $\mathbb{Z}_2$-characteristic vector of $F_i$ for all $0\le i\le n$. Considering the vertex $v_0=\bigcap_{j=1}^n F_j$, we have $\{b_i\ |\ 1\le i\le n\}$ as a basis of $\Z$. Since it is evident that any $n$ $\mathbb{Z}_2$-characteristic vectors are linearly independent, we get that $b_0=\sum_{j=1}^n b_j$. Fix an order of the elements of $\Z$ and let $g_i$ denote the $i^{th}$ element of $\Z$. Now, let us denote $g_i\times P^n$ by $t_i$ for all $1\le i\le 2^n$. By the construction of $M^n(\lambda)=\Z\times P^n/\sim$, the faces $F_j^k$ and $F_j^l$ of $t_k$ and $t_l$, respectively, are identified if and only if $g_k+g_l=b_j$ for $0\le j\le n$ and $1\le k,l\le 2^n$. Since $t_i$ is the $n$-simplex $[v_0^i,v_1^i,\dots,v_n^i]$, we color its vertices as $0,1,\dots,n$ in order, for all $1\le i\le 2^n$. Therefore, $M^n(\lambda)$ is an $(n+1)$-colorable simplicial cell complex, and hence, it can be represented by an $(n+1)$-regular colored graph $(\Gamma,\gamma)$ (cf. subsection \ref{crystal}). Clearly, this $(n+1)$-colorable simplicial cell complex has exactly $n+1$ vertices. It is known that the small cover $M^n(\lambda)$ over $\Delta_n$ is $\mathbb{RP}^n$ (cf. \cite{dj91}). Therefore, we get a crystallization ($\Gamma,\gamma$) of $\mathbb{RP}^n$ with $2^n$ vertices. This $2^n$-vertex crystallization of $\mathbb{RP}^n$ is known in the literature (cf. \cite{cs07}). In this section, we establish the uniqueness up to isomorphism of the crystallization of \( \mathbb{RP}^n \) with \( 2^n \) vertices for every \( n \geq 2 \).

\begin{proposition}[\cite{cs07}]\label{min}
     For $n \geq 2$, the number of vertices in a crystallization of $\mathbb{RP}^n$ is at least $2^n$.
\end{proposition}

\begin{proposition}[\cite{cs07}]\label{length}
For $n \geq 2$, if $(\Gamma,\gamma)$ is a crystallization of $\mathbb{RP}^n$  with $2^n$ vertices, then the length of every bi-colored cycle of $\Gamma$ is at least $4$.
\end{proposition}

\begin{lemma}\label{i-simplex}
    The number of $i$-simplices, $1\leq i < n$, colored by $\{k_0,k_1,\dots,k_i\}\subset \Delta_n$ in a contracted triangulation of $\mathbb{RP}^n$ is at least $2^i$. 
\end{lemma}
\begin{proof}
Let $K$ be a contracted triangulation of $\mathbb{RP}^n$.
Let $j\in \Delta_n \backslash \{k_0,k_1,\dots,k_i\}$. The disjoint star of $v_j$ (unique vertex colored by color $j$) in $K$ is an $n$-ball whose boundary is the disjoint link of $v_j$ in $K$, which is an $(n-1)$-sphere. Then, $K$ is obtained from the disjoint star of $v_j$ by identifying the $(n-1)$-simplices of the boundary of the disjoint star of $v_j$ pairwise. Since $\mathbb{RP}^{n-1}$ is a spine of $\mathbb{RP}^{n}$ (i.e., $\mathbb{RP}^{n}-\mathbb{B}^n$ deformation retracts to $\mathbb{RP}^{n-1}$), the induced identifications on the disjoint link of $v_j$ will give us a quotient complex $M$ that deformation retracts to a contracted triangulation of $\mathbb{RP}^{n-1}$. Therefore, the number of $i$-simplices colored by $\{k_0,k_1,\dots,k_i\}$ in the given contracted triangulation of $\mathbb{RP}^n$ is greater than or equal to the number of $i$-simplices  colored by  $\{k_0,k_1,\dots,k_i\}$ in the induced contracted triangulation of $\mathbb{RP}^{n-1}.$ 
By iteratively applying the same argument, we see that the number of $i$-simplices colored by $\{k_0,k_1,\dots,k_i\}$ in the given contracted triangulation of $\mathbb{RP}^n$ is greater than or equal to the number of $i$-simplices  colored by  $\{k_0,k_1,\dots,k_i\}$ in the induced contracted triangulation of $\mathbb{RP}^{i}.$
According to Proposition \ref{min}, this number is at least $2^{i}.$ This completes the proof.
\end{proof}

\begin{corollary}\label{corollary:cycle}
Let $(\Gamma,\gamma)$ be a crystallization of $\mathbb{RP}^n$ with $2^n$ vertices. Then, the following properties hold:
\begin{enumerate}[$(a)$]
\item For any subset $\{k_0,k_1,\dots,k_i\}\subset \Delta_n$, we have $g_{\{k_0,k_1,\dots,k_i\}}=2^{n-i-1}$.
\item Every bi-colored cycle in $\Gamma$  has length $4$.
\end{enumerate}   
\end{corollary}

\begin{proof}
Part \((a)\) follows directly from the proof of Lemma \ref{i-simplex}, while Part \((b)\) follows from Part \((a)\) and Proposition \ref{length}.
\end{proof}

\begin{theorem}\label{theorem:unique}
There exists a unique crystallization of $\mathbb{RP}^n$ with $2^n$ vertices for every $n\ge 2$.
\end{theorem}

\begin{proof}
Let \((\Gamma, \gamma)\) be a crystallization of \(\mathbb{RP}^n\) with \(2^n\) vertices. Due to Corollary \ref{corollary:cycle}, \(g_{\{0,1\}} = 2^{n-2}\), i.e., we have \(2^{n-2}\) disjoint 2-cubes in $\Gamma_{\{0,1\}}$. If $n=2$, then by Corollary \ref{corollary:cycle}(b), the $2$-colored edges can only join diagonal vertices of the $2$-cube. Hence, we get that $(\Gamma,\gamma)$ is the unique crystallization of $\mathbb{RP}^2$. If $n=3$, since $\Gamma$ is a crystallization, an edge of color $2$ joins the vertices of the two components of $\Gamma_{\{0,1\}}$. By Corollary \ref{corollary:cycle}(b), $\Gamma_{\{0,1,2\}}$ is a $3$-cube. If a $3$-colored edge joins the diagonal vertices of a $2$-face of this $3$-cube, then by Corollary \ref{corollary:cycle}(b), we get a representation of $\mathbb{RP}^2$. This implies that the disjoint link of a vertex is $\mathbb{RP}^2$, which is not possible. So, an edge of color $3$ must join two diagonal vertices of this $3$-cube. Hence, we get that $(\Gamma,\gamma)$ is the unique crystallization of $\mathbb{RP}^3$.   

If $n> 3$, note that within any of these $2$-cubes, an edge of color $2$ cannot exist. If such an edge were present, then by Corollary \ref{corollary:cycle}$(b)$, one component of $\Gamma_{\{0,1,2\}}$ would represent \(\mathbb{RP}^2\). This leads to a contradiction as it would imply that the disjoint link of an \((n-3)\)-simplex (with vertices colored by \(\Delta_n \backslash \{0,1,2\}\)) is \(\mathbb{RP}^2\). 
So, an edge of color $2$ can only be between the vertices of two different components of $\Gamma_{\{0,1\}}$. By Corollary \ref{corollary:cycle}$(b)$, two components of $\Gamma_{\{0,1\}}$ will form a $3$-cube in $\Gamma_{\{0,1,2\}}$. Also, by Corollary \ref{corollary:cycle}$(a)$, we have $g_{\{0,1,2\}}=2^{n-3}.$ Thus, we get $2^{n-3}$ $3$-cubes in $\Gamma_{\{0,1,2\}}$, each having edges colored by $\{0,1,2\}$. Now, again note that an edge of color $3$ cannot exist within a $3$-cube. By Corollary \ref{corollary:cycle}$(b)$, the presence of such an edge would result in a representation of either \(\mathbb{RP}^3\) or \(\mathbb{RP}^2\), depending on whether the $3$-colored edge is a diagonal of the $3$-cube or lies on one of its faces. Thus, an edge of color $3$ can only be between the vertices of two different components of \(\Gamma_{\{0,1,2\}}\).
Thus, proceeding in the same manner, we finally get two $(n-1)$-cubes whose edges are colored by $\{0,1,\dots,n-2\}$. Since $\Gamma$ is a crystallization, $\Gamma_{\hat{n}}$ is connected. Hence, by Corollary \ref{corollary:cycle}(b), $\Gamma_{\hat{n}}$ is an $n$-cube. Clearly, this $n$-cube represents $\mathbb S^{n-1}$, and it is obtained uniquely.
Again, note that if we connect two vertices by an edge of color $n$ which are present diagonally in an $r$-cube whose edges are colored by $\{k_0,k_1,\dots,k_{r-1}\} \subset \{0,1,\dots,n-1\},$ where $2\le r\le n-1 $, then by Corollary \ref{corollary:cycle}$(b)$, other $n$-colored edges, incident to the vertices of this $r$-cube, will have to be in the same $r$-cube. This results in a representation of $\mathbb{RP}^{r}$, which implies that the disjoint link of an $(n-r-1)$-simplex is $\mathbb{RP}^{r}$. So, there is a unique choice for the vertices of an edge of color $n$. The endpoints of every $n$-colored edge are the diagonal vertices of the $n$-cube. Thus, this $(\Gamma,\gamma)$ is unique.
\end{proof}

\subsection{On the regular genus of $\mathbb{RP}^{n-1}$-bundles over $\mathbb{S}^1$} 

\noindent \textbf{\underline{Construction of a gem of the small cover $M^n(\lambda)$ over $\Delta^{n-1} \times I$:}} 
The simple polytope $\Delta^{n-1} \times I$ has $2n$ vertices, where $n\ge 3$. Let us denote the simplex $\Delta^{n-1}$ by $[a_0,a_1,\dots,a_{n-1}]$, the $(n-1)$-simplex $\Delta^{n-1} \times \{0\}$ by $[v_0,v_1,\dots,v_{n-1}]$ and $\Delta^{n-1} \times \{1\}$ by $[w_1,w_2,\dots,w_{n}]$, where $(a_i,0)=v_{i}$ and $(a_i,1)=w_{i+1}$ for all $0\le i \le n-1$. Clearly, the number of $(n-1)$-faces of the polytope $\Delta^{n-1} \times I$ is $n+2$, and let the set of these faces be 
$$\mathcal F=\bigl\{(v_0,\dots,v_{i-1}, \hat{v}_i, v_{i+1},\dots,v_{n-1},w_1,\dots,w_{i}, \hat{w}_{i+1}, w_{i+2},\dots,w_{n})=F_i,$$ $$\  [v_0,\dots,v_{n-1}]=F_{n},\ [w_1,\dots,w_{n}]=F_{n+1}\ |\  0\le i\le n-1 \bigr\},$$ 
where the `hat' symbol  $\,\hat{}\,$  over a vertex indicates that the vertex is deleted.
Now, let $\lambda:\mathcal F \to \Z$ be a $\mathbb Z_2$-characteristic function. Then, from Subsection \ref{smallcover}, we have the small cover $M^n(\lambda)=\Z \times (\Delta^{n-1} \times I)/\sim$ over $\Delta^{n-1} \times I$. Fix an order of the elements of $\Z$ and let $g_i$ denote the $i^{th}$ element of $\Z$. Let us denote $g_i\times (\Delta^{n-1} \times I)$ by $t_i$ for all $1\le i\le 2^n$. Now, color the vertices \( v_0^i, v_1^i, \dots, v_{n-1}^i \) with the labels (colors) \( 0, 1, \dots, n-1 \) in order, and the vertices \( w_1^i, w_2^i, \dots, w_n^i \) with the labels \( 1, 2, \dots, n \) in order, for all $1\le i\le 2^n$. Using the standard triangulation technique, we get an $(n+1)$-colorable simplicial cell complex with boundary whose geometric carrier is $t_i$. In this colored triangulation of $t_i$, the number of $n$-simplices is $n$. These $n$-simplices are $t_i^1=[v_0^i,v_1^i,\dots,v_{n-1}^i,w_n^i]$, $t_i^2=[v_0^i,v_1^i,\dots,w_{n-1}^i,w_n^i],\dots$, $t_i^n=[v_0^i,w_1^i,\dots,w_{n-1}^i,w_n^i]$. Therefore, $M^n(\lambda)$ admits a colored triangulation, and hence $M^n(\lambda)$ can be represented by an $(n+1)$-regular colored graph $(\Gamma,\gamma)$. Clearly, in this colored triangulation of $M^n(\lambda)$, the number of $n$-simplices is $n2^n$. We denote the vertex of $\Gamma$ corresponding to the $n$-simplex $t_i^j$ by $T_i^j$ for all $1\le i \le 2^n$ and $1\le j\le n$. It is evident from the construction of $M^n(\lambda)$ that $g_{\hat{0}}=g_{\hat{n}}=1$ and $g_{\hat{j}}=2$ for all $1\le j\le n-1$.  

Note that, with a similar construction, given a simple $n$-polytope $P^n$, together with a $\mathbb Z_2$-characteristic function $\lambda:\mathcal F(P^n)\to \Z$, one can always obtain a gem of the small cover $M^n(\lambda)$, using the colored triangulation of $P^n$. We illustrate the above construction with an example of $\Delta^{3} \times I$.

\begin{figure}[ht]
\tikzstyle{ver}=[]
\tikzstyle{edge} = [draw,thick,-]
\centering
\begin{tikzpicture}[scale=1]

\node[ver] () at (-6,5){$t_i^1=[v_0^i,v_1^i,v_2^i,v_3^i,w_4^i]$}; 
\node[ver] () at (-6,4.3){$[v_0^i,v_1^i,v_2^i,v_3^i]$};
\node[ver] () at (-3,4.3){$[v_0^i,v_1^i,v_2^i,w_4^i]$};
\node[ver] () at (0,4.3){$[v_0^i,v_1^i,v_3^i,w_4^i]$};
\node[ver] () at (3,4.3){$[v_0^i,v_2^i,v_3^i,w_4^i]$};
\node[ver] () at (6,4.3){$[v_1^i,v_2^i,v_3^i,w_4^i]$};

\node[ver] () at (-6,3.8){\scriptsize{$(0,0,0,1)$}};
\node[ver] () at (0,3.8){\scriptsize{$(0,0,1,0)$}};
\node[ver] () at (3,3.8){\scriptsize{$(0,1,0,0)$}};
\node[ver] () at (6,3.8){\scriptsize{$(1,0,0,0)$}};

\node[ver] () at (-6,3){$t_i^2=[v_0^i,v_1^i,v_2^i,w_3^i,w_4^i]$}; 
\node[ver] () at (-6,2.3){$[v_0^i,v_1^i,v_2^i,w_3^i]$};
\node[ver] () at (-3,2.3){$[v_0^i,v_1^i,v_2^i,w_4^i]$};
\node[ver] () at (0,2.3){$[v_0^i,v_1^i,w_3^i,w_4^i]$};
\node[ver] () at (3,2.3){$[v_0^i,v_2^i,w_3^i,w_4^i]$};
\node[ver] () at (6,2.3){$[v_1^i,v_2^i,w_3^i,w_4^i]$};

\node[ver] () at (-6,1.8){\scriptsize{$(1,1,1,0)$}};
\node[ver] () at (3,1.8){\scriptsize{$(0,1,0,0)$}};
\node[ver] () at (6,1.8){\scriptsize{$(1,0,0,0)$}};

\node[ver] () at (-6,1){$t_i^3=[v_0^i,v_1^i,w_2^i,w_3^i,w_4^i]$}; 
\node[ver] () at (-6,0.3){$[v_0^i,v_1^i,w_2^i,w_3^i]$};
\node[ver] () at (-3,0.3){$[v_0^i,v_1^i,w_2^i,w_4^i]$};
\node[ver] () at (0,0.3){$[v_0^i,v_1^i,w_3^i,w_4^i]$};
\node[ver] () at (3,0.3){$[v_0^i,w_2^i,w_3^i,w_4^i]$};
\node[ver] () at (6,0.3){$[v_1^i,w_2^i,w_3^i,w_4^i]$};

\node[ver] () at (-6,-0.2){\scriptsize{$(1,1,1,0)$}};
\node[ver] () at (-3,-0.2){\scriptsize{$(0,0,1,0)$}};
\node[ver] () at (6,-0.2){\scriptsize{$(1,0,0,0)$}};

\node[ver] () at (-6,-1){$t_i^4=[v_0^i,w_1^i,w_2^i,w_3^i,w_4^i]$}; 
\node[ver] () at (-6,-1.7){$[v_0^i,w_1^i,w_2^i,w_3^i]$};
\node[ver] () at (-3,-1.7){$[v_0^i,w_1^i,w_2^i,w_4^i]$};
\node[ver] () at (0,-1.7){$[v_0^i,w_1^i,w_3^i,w_4^i]$};
\node[ver] () at (3,-1.7){$[v_0^i,w_2^i,w_3^i,w_4^i]$};
\node[ver] () at (6,-1.7){$[w_1^i,w_2^i,w_3^i,w_4^i]$};

\node[ver] () at (-6,-2.2){\scriptsize{$(1,1,1,0)$}};
\node[ver] () at (-3,-2.2){\scriptsize{$(0,0,1,0)$}};
\node[ver] () at (0,-2.2){\scriptsize{$(0,1,0,0)$}};
\node[ver] () at (6,-2.2){\scriptsize{$(0,0,0,1)$}};

\path[edge] (-3,4)--(-3,2.5);
\path[edge] (0,2)--(0,0.5);
\path[edge] (3,0)--(3,-1.5);

\node[ver] () at (-2.2,3.4){\scriptsize{identified}};
\node[ver] () at (0.8,1.4){\scriptsize{identified}};
\node[ver] () at (3.8,-0.6){\scriptsize{identified}};

\end{tikzpicture}
\caption{$t_i^j$ with all its $3$-faces and their $\mathbb Z_2$-characteristic vectors for all $1\le j\le 4$. }\label{table}
\end{figure}

\begin{example}\label{ex}
{\rm Let $\lambda: \mathcal F(\Delta^{3} \times I) \to \mathbb Z_2^4$ be defined as $\lambda(F_0)=(1,0,0,0),\ \lambda(F_1)=(0,1,0,0),\ \lambda(F_2) \\=(0,0,1,0),\ \lambda(F_3)=(1,1,1,0),\ \lambda(F_4)=(0,0,0,1),\ \lambda(F_5)=(0,0,0,1)$. Clearly, $\lambda$ is a $Z_2$-characteristic function and $M^4(\lambda)=\mathbb{RP}^3\times \mathbb{S}^1$ (cf. Proposition \ref{Proposition: product}). Now, in $\mathbb Z_2^4\times (\Delta^3\times I)$, we have $16$ copies of $\Delta^{3} \times I$, and in each copy, the number of $4$-simplices is $4$. Let us fix an order of the elements of $\mathbb{Z}_2^4$ as follows: $g_1=(1,0,0,0)$, $g_2=(0,1,0,0)$, $g_3=(0,0,1,0)$, $g_4=(0,0,0,1)$, $g_5=(0,1,1,1)$, $g_6=(1,0,1,1)$, $g_7=(1,1,0,1)$, $g_8=(1,1,1,0)$, $g_9=(0,0,1,1)$, $g_{10}=(0,1,0,1)$, $g_{11}=(0,1,1,0)$, $g_{12}=(1,0,0,1)$, $g_{13}=(1,0,1,0)$, $g_{14}=(1,1,0,0)$, $g_{15}=(1,1,1,1)$, $g_{16}=(0,0,0,0)$. Let the $4$-simplices of $t_i$ be $t_i^1=[v_0^i,v_1^i,v_2^i,v_3^i,w_4^i]$, $t_i^2=[v_0^i,v_1^i,v_2^i,w_3^i,w_4^i]$, $t_i^3=[v_0^i,v_1^i,w_2^i,w_3^i,w_4^i]$, $t_i^4=[v_0^i,w_1^i,w_2^i,w_3^i,w_4^i]$, and we will denote the vertex of the gem $(\Gamma,\gamma)$ corresponding to the $4$-simplex $t_i^j$ by $T_i^j$ for all $1\le i\le 16$ and $1\le j\le 4$. In Figure \ref{table}, for $t_i$, where $1\le i \le 16$, all the $3$-faces of $t_i^j,$ are written, and below them, their corresponding $\mathbb{Z}_2$-characteristic vectors are written for all $1\le j\le 4$. Figure $\ref{fig1} (a)$ exhibits the gem $(\Gamma,\gamma)$ of $M^4(\lambda)$, a small cover over $\Delta^3 \times I.$}
\end{example}

 \begin{figure}[h!]
\tikzstyle{ver}=[]
\tikzstyle{verti}=[circle, draw, fill=black!100, inner sep=0pt, minimum width=3pt]
\tikzstyle{edge} = [draw,thick,-]
    \centering
\begin{tikzpicture}[scale=0.4]
\begin{scope}[shift={(0,40)}]
\begin{scope}[shift={(-16,16)}]
\foreach \x/\y/\z in
{1.5/1.5/1,1.5/-1.5/2,-1.5/-1.5/3,-1.5/1.5/4,3/3/5,3/-3/6,-3/-3/7,-3/3/8,4.5/4.5/9,4.5/-4.5/10,-4.5/-4.5/11,-4.5/4.5/12,6/6/13,6/-6/14,-6/-6/15,-6/6/16}
{\node[verti] (a\z) at (\x,\y){};}

\foreach \x/\y/\z in
{1.1/2/12,1.1/-2/4,-1.1/-2/16,-1.1/2/1,2.6/3.5/7,2.6/-3.5/10,-2.6/-3.5/2,-2.6/3.5/14,4.1/5/15,4.1/-5/5,-4.1/-5/11,-4.1/5/8,5.6/6.5/6,5.6/-6.5/9,-5.6/-6.5/3,-5.6/6.5/13}
{\node[ver] () at (\x,\y){\tiny{$T_{\z}^1$}};}

\foreach \x/\y in 
{5/9,6/10,7/11,8/12,1/4,2/3,6/7,8/5,10/11,12/9,14/15,16/13}
{\path[edge] (a\x)--(a\y);}

\foreach \x/\y in 
{1/4,2/3,6/7,8/5,10/11,12/9,14/15,16/13}
{\draw [line width=2pt, line cap=round, dash pattern=on 0pt off 1.3\pgflinewidth]  (a\x) -- (a\y);}

\foreach \x/\y in 
{1/2,4/3,6/5,8/7,10/9,12/11,14/13,16/15}
{\draw[decorate,decoration={snake, amplitude=1pt, segment length=8pt}]
(a\x) -- (a\y);}

\foreach \x/\y in 
{5/1,6/2,7/3,8/4,9/13,10/14,11/15,12/16}
{\path[edge,dotted] (a\x)--(a\y);}

\draw[edge] plot [smooth,tension=0.5] coordinates{(a1) (4.2,2.5) (a13)};

\draw[edge] plot [smooth,tension=0.5] coordinates{(a2) (4.2,-2.5) (a14)};

\draw[edge] plot [smooth,tension=0.5] coordinates{(a3) (-4.2,-2.5) (a15)};

\draw[edge] plot [smooth,tension=0.5] coordinates{(a4) (-4.2,2.5) (a16)};

\end{scope}

\begin{scope}[shift={(4,17)}]
\begin{scope}[shift={(-5,0)}]
\foreach \x/\y/\z in
{1.5/1.5/1,1.5/-1.5/2,-1.5/-1.5/3,-1.5/1.5/4,3/3/5,3/-3/6,-3/-3/7,-3/3/8}
{\node[verti] (a\z) at (\x,\y){};}

\foreach \x/\y/\z in
{1.1/2/11,1.1/-2/8,-1.1/-2/16,-1.1/2/1,2.6/3.5/3,2.6/-3.5/13,-2.6/-3.5/2,-2.6/3.5/14}
{\node[ver] () at (\x,\y){\tiny{$T_{\z}^2$}};}

\foreach \x/\y in 
{1/4,2/3,6/7,8/5}
{\path[edge] (a\x)--(a\y);}

\foreach \x/\y in 
{1/4,2/3,6/7,8/5}
{\draw [line width=2pt, line cap=round, dash pattern=on 0pt off 1.3\pgflinewidth]  (a\x) -- (a\y);}

\foreach \x/\y in 
{1/2,4/3,6/5,8/7}
{\draw[decorate,decoration={snake, amplitude=1pt, segment length=8pt}]
(a\x) -- (a\y);}

\foreach \x/\y in 
{5/1,6/2,7/3,8/4}
{\path[edge,dotted] (a\x)--(a\y);}

\end{scope}

\begin{scope}[shift={(5,0)}]
\foreach \x/\y/\z in
{1.5/1.5/1,1.5/-1.5/2,-1.5/-1.5/3,-1.5/1.5/4,3/3/5,3/-3/6,-3/-3/7,-3/3/8}
{\node[verti] (a\z) at (\x,\y){};}

\foreach \x/\y/\z in
{1.1/2/15,1.1/-2/5,-1.1/-2/12,-1.1/2/4,2.6/3.5/6,2.6/-3.5/9,-2.6/-3.5/7,-2.6/3.5/10}
{\node[ver] () at (\x,\y){\tiny{$T_{\z}^2$}};}

\foreach \x/\y in 
{1/4,2/3,6/7,8/5}
{\path[edge] (a\x)--(a\y);}

\foreach \x/\y in 
{1/4,2/3,6/7,8/5}
{\draw [line width=2pt, line cap=round, dash pattern=on 0pt off 1.3\pgflinewidth]  (a\x) -- (a\y);}

\foreach \x/\y in 
{1/2,4/3,6/5,8/7}
{\draw[decorate,decoration={snake, amplitude=1pt, segment length=8pt}]
(a\x) -- (a\y);}

\foreach \x/\y in 
{5/1,6/2,7/3,8/4}
{\path[edge,dotted] (a\x)--(a\y);}

\end{scope}

\end{scope}

\begin{scope}[shift={(-14,4.5)}]
\begin{scope}[shift={(-5,0)}]
\foreach \x/\y/\z in
{1.5/1.5/1,1.5/-1.5/2,-1.5/-1.5/3,-1.5/1.5/4,3/3/5,3/-3/6,-3/-3/7,-3/3/8}
{\node[verti] (a\z) at (\x,\y){};}

\foreach \x/\y/\z in
{1.1/2/11,1.1/-2/8,-1.1/-2/16,-1.1/2/1,2.6/3.5/2,2.6/-3.5/14,-2.6/-3.5/3,-2.6/3.5/13}
{\node[ver] () at (\x,\y){\tiny{$T_{\z}^3$}};}

\foreach \x/\y in 
{1/4,2/3,6/7,8/5}
{\path[edge] (a\x)--(a\y);}

\foreach \x/\y in 
{1/4,2/3,6/7,8/5}
{\draw [line width=2pt, line cap=round, dash pattern=on 0pt off 1.3\pgflinewidth]  (a\x) -- (a\y);}

\foreach \x/\y in 
{1/2,4/3,6/5,8/7}
{\draw[decorate,decoration={snake, amplitude=1pt, segment length=8pt}]
(a\x) -- (a\y);}

\foreach \x/\y in 
{5/1,6/2,7/3,8/4}
{\path[edge,dashed] (a\x)--(a\y);}

\end{scope}

\begin{scope}[shift={(5,0)}]
\foreach \x/\y/\z in
{1.5/1.5/1,1.5/-1.5/2,-1.5/-1.5/3,-1.5/1.5/4,3/3/5,3/-3/6,-3/-3/7,-3/3/8}
{\node[verti] (a\z) at (\x,\y){};}

\foreach \x/\y/\z in
{1.1/2/15,1.1/-2/5,-1.1/-2/12,-1.1/2/4,2.6/3.5/7,2.6/-3.5/10,-2.6/-3.5/6,-2.6/3.5/9}
{\node[ver] () at (\x,\y){\tiny{$T_{\z}^3$}};}

\foreach \x/\y in 
{1/4,2/3,6/7,8/5}
{\path[edge] (a\x)--(a\y);}

\foreach \x/\y in 
{1/4,2/3,6/7,8/5}
{\draw [line width=2pt, line cap=round, dash pattern=on 0pt off 1.3\pgflinewidth]  (a\x) -- (a\y);}

\foreach \x/\y in 
{1/2,4/3,6/5,8/7}
{\draw[decorate,decoration={snake, amplitude=1pt, segment length=8pt}]
(a\x) -- (a\y);}

\foreach \x/\y in 
{5/1,6/2,7/3,8/4}
{\path[edge,dashed] (a\x)--(a\y);}
\end{scope}

\end{scope}

\begin{scope}[shift={(6,5.5)}]
\foreach \x/\y/\z in
{1.5/1.5/1,1.5/-1.5/2,-1.5/-1.5/3,-1.5/1.5/4,3/3/5,3/-3/6,-3/-3/7,-3/3/8,4.5/4.5/9,4.5/-4.5/10,-4.5/-4.5/11,-4.5/4.5/12,6/6/13,6/-6/14,-6/-6/15,-6/6/16}
{\node[verti] (a\z) at (\x,\y){};}

\foreach \x/\y/\z in
{1.1/2/11,1.1/-2/5,-1.1/-2/12,-1.1/2/1,2.6/3.5/3,2.6/-3.5/9,-2.6/-3.5/7,-2.6/3.5/14,4.1/5/16,4.1/-5/4,-4.1/-5/15,-4.1/5/8,5.6/6.5/2,5.6/-6.5/10,-5.6/-6.5/6,-5.6/6.5/13}
{\node[ver] () at (\x,\y){\tiny{$T_{\z}^4$}};}

\foreach \x/\y in 
{1/4,2/3,6/7,8/5,10/11,12/9,14/15,16/13}
{\path[edge] (a\x)--(a\y);}

\foreach \x/\y in 
{5/9,6/10,7/11,8/12}
{\path[edge,dashed] (a\x)--(a\y);}

\foreach \x/\y in 
{1/4,2/3,6/7,8/5,10/11,12/9,14/15,16/13}
{\draw [line width=2pt, line cap=round, dash pattern=on 0pt off 1.3\pgflinewidth]  (a\x) -- (a\y);}

\foreach \x/\y in 
{1/2,4/3,6/5,8/7,10/9,12/11,14/13,16/15}
{\draw[decorate,decoration={snake, amplitude=1pt, segment length=8pt}]
(a\x) -- (a\y);}

\foreach \x/\y in 
{5/1,6/2,7/3,8/4,9/13,10/14,11/15,12/16}
{\path[edge] (a\x)--(a\y);}

\draw[edge,dashed] plot [smooth,tension=0.5] coordinates{(a1) (4.2,2.5) (a13)};

\draw[edge,dashed] plot [smooth,tension=0.5] coordinates{(a2) (4.2,-2.5) (a14)};

\draw[edge,dashed] plot [smooth,tension=0.5] coordinates{(a3) (-4.2,-2.5) (a15)};

\draw[edge,dashed] plot [smooth,tension=0.5] coordinates{(a4) (-4.2,2.5) (a16)};

\end{scope}

\begin{scope}[shift={(-5,11)}]
\foreach \x/\y/\z/\w in 
{-1/1/T_i^1/1,2/1/T_i^2/2,2/-1.5/T_i^4/4,-1/-1.5/T_i^3/3}
{\node[ver] (\w) at (\x,\y){\tiny{$\z$}};}

\path[edge,dashed] (1)--(2);
\path[edge] (2)--(3);
\path[edge,dotted] (3)--(4);

\end{scope}
\node[ver] () at (-12,-0.5){$(a)$ }; 
\end{scope}

\begin{scope}[shift={(0,15.5)}]
\begin{scope}[shift={(-13,16)}]
\foreach \x/\y/\z in
{1.5/1.5/1,1.5/-1.5/2,-1.5/-1.5/3,-1.5/1.5/4,3/3/5,3/-3/6,-3/-3/7,-3/3/8,4.5/4.5/9,4.5/-4.5/10,-6/-4.5/11,-6/4.5/12,6/6/13,6/-6/14,-7.5/-6/15,-7.5/6/16,-9/-6/17,-9/6/18,-4.5/-3/19,-4.5/3/20}
{\node[verti] (a\z) at (\x,\y){};}

\foreach \x/\y/\z in
{1.1/2/T_{12}^1,1.1/-2/T_4^1,-1.1/-2/T_{8}^2,-1.1/2/T_{11}^2,2.6/3.5/T_{7}^1,2.6/-3.5/T_{10}^1,-2.6/-3.5/T_{13}^2,-2.6/3.5/T_{3}^2,4.1/5/T_{15}^1,4.1/-5/T_{5}^1,-5.6/-5/T_{11}^1,-5.6/5/T_{8}^1,5.6/6.5/T_{6}^1,5.6/-6.5/T_{9}^1,-7.1/-6.5/T_{3}^1,-7.1/6.5/T_{13}^1,-9.7/-6/T_{16}^3,-9.7/6/T_1^3,-3.8/-2.8/T_2^3,-3.8/2.8/T_{14}^3}
{\node[ver] () at (\x,\y){\tiny{$\z$}};}

\foreach \x/\y in 
{5/9,6/10,19/11,20/12,1/4,2/3,6/7,8/5,10/11,12/9,14/15,16/13,15/17,16/18}
{\path[edge] (a\x)--(a\y);}

\foreach \x/\y in 
{1/4,2/3,6/7,8/5,10/11,12/9,14/15,16/13}
{\draw [line width=2pt, line cap=round, dash pattern=on 0pt off 1.3\pgflinewidth]  (a\x) -- (a\y);}

\foreach \x/\y in 
{1/2,4/3,6/5,8/7,10/9,12/11,14/13,16/15,17/18,19/20}
{\draw[decorate,decoration={snake, amplitude=1pt, segment length=8pt}]
(a\x) -- (a\y);}

\foreach \x/\y in 
{5/1,6/2,7/3,8/4,9/13,10/14,11/15,12/16}
{\path[edge,dotted] (a\x)--(a\y);}

\draw[edge] plot [smooth,tension=0.5] coordinates{(a1) (4.2,2.5) (a13)};

\draw[edge] plot [smooth,tension=0.5] coordinates{(a2) (4.2,-2.5) (a14)};

\end{scope}

\begin{scope}[shift={(5.5,17)}]

\foreach \x/\y/\z in
{1.5/1.5/1,1.5/-1.5/2,-1.5/-1.5/3,-1.5/1.5/4,3/3/5,3/-3/6,-3/-3/7,-3/3/8}
{\node[verti] (a\z) at (\x,\y){};}

\foreach \x/\y/\z in
{1.1/2/15,1.1/-2/5,-1.1/-2/12,-1.1/2/4,2.6/3.5/6,2.6/-3.5/9,-2.6/-3.5/7,-2.6/3.5/10}
{\node[ver] () at (\x,\y){\tiny{$T_{\z}^2$}};}

\foreach \x/\y in 
{1/4,2/3,6/7,8/5}
{\path[edge] (a\x)--(a\y);}

\foreach \x/\y in 
{1/4,2/3,6/7,8/5}
{\draw [line width=2pt, line cap=round, dash pattern=on 0pt off 1.3\pgflinewidth]  (a\x) -- (a\y);}

\foreach \x/\y in 
{1/2,4/3,6/5,8/7}
{\draw[decorate,decoration={snake, amplitude=1pt, segment length=8pt}]
(a\x) -- (a\y);}

\foreach \x/\y in 
{5/1,6/2,7/3,8/4}
{\path[edge,dotted] (a\x)--(a\y);}

\end{scope}

\begin{scope}[shift={(-14,4.5)}]
\begin{scope}[shift={(-5,0)}]
\foreach \x/\y/\z in
{1.5/1.5/1,1.5/-1.5/2,-1.5/-1.5/3,-1.5/1.5/4,3/3/5,3/-3/6,-3/-3/7,-3/3/8}
{\node[verti] (a\z) at (\x,\y){};}

\foreach \x/\y/\z in
{1.1/2/11,1.1/-2/8,-1.1/-2/16,-1.1/2/1,2.6/3.5/2,2.6/-3.5/14,-2.6/-3.5/3,-2.6/3.5/13}
{\node[ver] () at (\x,\y){\tiny{$T_{\z}^3$}};}

\foreach \x/\y in 
{1/4,2/3,6/7,8/5}
{\path[edge] (a\x)--(a\y);}

\foreach \x/\y in 
{1/4,2/3,6/7,8/5}
{\draw [line width=2pt, line cap=round, dash pattern=on 0pt off 1.3\pgflinewidth]  (a\x) -- (a\y);}

\foreach \x/\y in 
{1/2,4/3,6/5,8/7}
{\draw[decorate,decoration={snake, amplitude=1pt, segment length=8pt}]
(a\x) -- (a\y);}

\foreach \x/\y in 
{5/1,6/2,7/3,8/4}
{\path[edge,dashed] (a\x)--(a\y);}

\end{scope}

\begin{scope}[shift={(5,0)}]
\foreach \x/\y/\z in
{1.5/1.5/1,1.5/-1.5/2,-1.5/-1.5/3,-1.5/1.5/4,3/3/5,3/-3/6,-3/-3/7,-3/3/8}
{\node[verti] (a\z) at (\x,\y){};}

\foreach \x/\y/\z in
{1.1/2/15,1.1/-2/5,-1.1/-2/12,-1.1/2/4,2.6/3.5/7,2.6/-3.5/10,-2.6/-3.5/6,-2.6/3.5/9}
{\node[ver] () at (\x,\y){\tiny{$T_{\z}^3$}};}

\foreach \x/\y in 
{1/4,2/3,6/7,8/5}
{\path[edge] (a\x)--(a\y);}

\foreach \x/\y in 
{1/4,2/3,6/7,8/5}
{\draw [line width=2pt, line cap=round, dash pattern=on 0pt off 1.3\pgflinewidth]  (a\x) -- (a\y);}

\foreach \x/\y in 
{1/2,4/3,6/5,8/7}
{\draw[decorate,decoration={snake, amplitude=1pt, segment length=8pt}]
(a\x) -- (a\y);}

\foreach \x/\y in 
{5/1,6/2,7/3,8/4}
{\path[edge,dashed] (a\x)--(a\y);}
\end{scope}

\end{scope}

\begin{scope}[shift={(6,5)}]
\foreach \x/\y/\z in
{1.5/1.5/1,1.5/-1.5/2,-1.5/-1.5/3,-1.5/1.5/4,3/3/5,3/-3/6,-3/-3/7,-3/3/8,4.5/4.5/9,4.5/-4.5/10,-4.5/-4.5/11,-4.5/4.5/12,6/6/13,6/-6/14,-6/-6/15,-6/6/16}
{\node[verti] (a\z) at (\x,\y){};}

\foreach \x/\y/\z in
{1.1/2/11,1.1/-2/5,-1.1/-2/12,-1.1/2/1,2.6/3.5/3,2.6/-3.5/9,-2.6/-3.5/7,-2.6/3.5/14,4.1/5/16,4.1/-5/4,-4.1/-5/15,-4.1/5/8,5.6/6.5/2,5.6/-6.5/10,-5.6/-6.5/6,-5.6/6.5/13}
{\node[ver] () at (\x,\y){\tiny{$T_{\z}^4$}};}

\foreach \x/\y in 
{1/4,2/3,6/7,8/5,10/11,12/9,14/15,16/13}
{\path[edge] (a\x)--(a\y);}

\foreach \x/\y in 
{5/9,6/10,7/11,8/12}
{\path[edge,dashed] (a\x)--(a\y);}

\foreach \x/\y in 
{1/4,2/3,6/7,8/5,10/11,12/9,14/15,16/13}
{\draw [line width=2pt, line cap=round, dash pattern=on 0pt off 1.3\pgflinewidth]  (a\x) -- (a\y);}

\foreach \x/\y in 
{1/2,4/3,6/5,8/7,10/9,12/11,14/13,16/15}
{\draw[decorate,decoration={snake, amplitude=1pt, segment length=8pt}]
(a\x) -- (a\y);}

\foreach \x/\y in 
{5/1,6/2,7/3,8/4,9/13,10/14,11/15,12/16}
{\path[edge] (a\x)--(a\y);}

\draw[edge,dashed] plot [smooth,tension=0.5] coordinates{(a1) (4.2,2.5) (a13)};

\draw[edge,dashed] plot [smooth,tension=0.5] coordinates{(a2) (4.2,-2.5) (a14)};

\draw[edge,dashed] plot [smooth,tension=0.5] coordinates{(a3) (-4.2,-2.5) (a15)};

\draw[edge,dashed] plot [smooth,tension=0.5] coordinates{(a4) (-4.2,2.5) (a16)};

\end{scope}

\begin{scope}[shift={(-4,10)}]
\foreach \x/\y/\z/\w in 
{-1/1/T_i^1/1,2/1/T_i^2/2,2/-1.5/T_i^4/4,-1/-1.5/T_i^3/3}
{\node[ver] (\w) at (\x,\y){\tiny{$\z$}};}

\path[edge,dashed] (1)--(2);
\path[edge] (2)--(3);
\path[edge,dotted] (3)--(4);

\end{scope}

\node[ver] () at (-12,-0.5){$(b)$}; 
\end{scope}

\begin{scope} [shift = {(-15,13)}]
\foreach \x/\y/\z in {1/-0.5/0,5/-0.5/1,9/-0.5/2,13/-0.5/3,17/-0.5/4}
{\node[ver] () at (\x,\y){$\z$};}
\path[edge,dashed] (12,0) -- (14,0);
\path[edge] (16,0) -- (18,0);
\draw [line width=2pt, line cap=round, dash pattern=on 0pt off 1.3\pgflinewidth]  (16,0) -- (18,0);
\path[edge] (8,0) -- (10,0);
\draw[decorate,decoration={snake, amplitude=1pt, segment length=8pt}] (0,0) -- (2,0);

\path[edge,dotted] (4,0) -- (6,0);
\end{scope} 

\end{tikzpicture}
\caption{$(a)$ The gem $\Gamma$ of $\mathbb{RP}^3\times I$ with $64$ vertices, $(b)$ The gem $\Gamma^1$ of $\mathbb{RP}^3\times I$ with $56$ vertices.}\label{fig1}
\end{figure}

 \begin{figure}[h!]
\tikzstyle{ver}=[]
\tikzstyle{verti}=[circle, draw, fill=black!100, inner sep=0pt, minimum width=3pt]
\tikzstyle{edge} = [draw,thick,-]
    \centering
\begin{tikzpicture}[scale=0.4]
\begin{scope}[shift={(0,40)}]
\begin{scope}[shift={(-13,16)}]
\foreach \x/\y/\z in
{1.5/1.5/1,1.5/-1.5/2,-1.5/-1.5/3,-1.5/1.5/4,3/3/5,3/-3/6,-3/-3/7,-3/3/8,4.5/4.5/9,4.5/-4.5/10,-6/-4.5/11,-6/4.5/12,6/6/13,6/-6/14,-7.5/-6/15,-7.5/6/16,-9/-6/17,-9/6/18,-4.5/-3/19,-4.5/3/20}
{\node[verti] (a\z) at (\x,\y){};}

\foreach \x/\y/\z in
{1.1/2/T_{12}^1,1.1/-2/T_4^1,-1.1/-2/T_{8}^2,-1.1/2/T_{11}^2,2.6/3.5/T_{7}^1,2.6/-3.5/T_{10}^1,-2.6/-3.5/T_{13}^2,-2.6/3.5/T_{3}^2,4.1/5/T_{15}^1,4.1/-5/T_{5}^1,-5.6/-5/T_{11}^1,-5.6/5/T_{8}^1,5.6/6.5/T_{6}^1,5.6/-6.5/T_{9}^1,-7.1/-6.5/T_{3}^1,-7.1/6.5/T_{13}^1,-9.7/-6/T_{16}^3,-9.7/6/T_1^3,-3.8/-2.8/T_2^3,-3.8/2.8/T_{14}^3}
{\node[ver] () at (\x,\y){\tiny{$\z$}};}

\foreach \x/\y in 
{5/9,6/10,19/11,20/12,1/4,2/3,6/7,8/5,10/11,12/9,14/15,16/13,15/17,16/18}
{\path[edge] (a\x)--(a\y);}

\foreach \x/\y in 
{1/4,2/3,6/7,8/5,10/11,12/9,14/15,16/13}
{\draw [line width=2pt, line cap=round, dash pattern=on 0pt off 1.3\pgflinewidth]  (a\x) -- (a\y);}

\foreach \x/\y in 
{1/2,4/3,6/5,8/7,10/9,12/11,14/13,16/15,17/18,19/20}
{\draw[decorate,decoration={snake, amplitude=1pt, segment length=8pt}]
(a\x) -- (a\y);}

\foreach \x/\y in 
{5/1,6/2,7/3,8/4,9/13,10/14,11/15,12/16}
{\path[edge,dotted] (a\x)--(a\y);}

\draw[edge] plot [smooth,tension=0.5] coordinates{(a1) (4.2,2.5) (a13)};

\draw[edge] plot [smooth,tension=0.5] coordinates{(a2) (4.2,-2.5) (a14)};

\end{scope}

\begin{scope}[shift={(5.5,18)}]

\foreach \x/\y/\z in
{1.5/1.5/1,1.5/-1.5/2,-1.5/-1.5/3,-1.5/1.5/4,3/3/5,3/-3/6,-3/-3/7,-3/3/8}
{\node[verti] (a\z) at (\x,\y){};}

\foreach \x/\y/\z in
{1.1/2/T_{15}^2,1.1/-2/T_{5}^2,-1.1/-2/T_{12}^2,-1.1/2/T_{4}^2,2.6/3.5/T_{6}^4,2.6/-3.5/T_{9}^4,-2.6/-3.5/T_{7}^4,-2.6/3.5/T_{10}^4}
{\node[ver] () at (\x,\y){\tiny{$\z$}};}

\foreach \x/\y in 
{1/4,2/3,6/7,8/5}
{\path[edge] (a\x)--(a\y);}

\foreach \x/\y in 
{1/4,2/3,6/7,8/5}
{\draw [line width=2pt, line cap=round, dash pattern=on 0pt off 1.3\pgflinewidth]  (a\x) -- (a\y);}

\foreach \x/\y in 
{1/2,4/3}
{\draw[decorate,decoration={snake, amplitude=1pt, segment length=8pt}]
(a\x) -- (a\y);}

\foreach \x/\y in 
{5/1,6/2,7/3,8/4}
{\path[edge,dotted] (a\x)--(a\y);}
\end{scope}

\begin{scope}[shift={(-14,4.5)}]
\begin{scope}[shift={(-5,0)}]
\foreach \x/\y/\z in
{1.5/1.5/1,1.5/-1.5/2,-1.5/-1.5/3,-1.5/1.5/4,3/3/5,3/-3/6,-3/-3/7,-3/3/8}
{\node[verti] (a\z) at (\x,\y){};}

\foreach \x/\y/\z in
{1.1/2/11,1.1/-2/8,-1.1/-2/16,-1.1/2/1,2.6/3.5/2,2.6/-3.5/14,-2.6/-3.5/3,-2.6/3.5/13}
{\node[ver] () at (\x,\y){\tiny{$T_{\z}^3$}};}

\foreach \x/\y in 
{1/4,2/3,6/7,8/5}
{\path[edge] (a\x)--(a\y);}

\foreach \x/\y in 
{1/4,2/3,6/7,8/5}
{\draw [line width=2pt, line cap=round, dash pattern=on 0pt off 1.3\pgflinewidth]  (a\x) -- (a\y);}

\foreach \x/\y in 
{1/2,4/3,6/5,8/7}
{\draw[decorate,decoration={snake, amplitude=1pt, segment length=8pt}]
(a\x) -- (a\y);}

\foreach \x/\y in 
{5/1,6/2,7/3,8/4}
{\path[edge,dashed] (a\x)--(a\y);}

\end{scope}

\begin{scope}[shift={(5,0)}]
\foreach \x/\y/\z in
{1.5/1.5/1,1.5/-1.5/2,-1.5/-1.5/3,-1.5/1.5/4,3/3/5,3/-3/6,-3/-3/7,-3/3/8}
{\node[verti] (a\z) at (\x,\y){};}

\foreach \x/\y/\z in
{1.1/2/T_{15}^3,1.1/-2/T_{5}^3,-1.1/-2/T_{12}^3,-1.1/2/T_{4}^3,2.6/3.5/T_{7}^1,2.6/-3.5/T_{10}^1,-2.6/-3.5/T_{6}^1,-2.6/3.5/T_{9}^1}
{\node[ver] () at (\x,\y){\tiny{$\z$}};}

\foreach \x/\y in 
{1/4,2/3}
{\path[edge] (a\x)--(a\y);}

\foreach \x/\y in 
{1/4,2/3}
{\draw [line width=2pt, line cap=round, dash pattern=on 0pt off 1.3\pgflinewidth]  (a\x) -- (a\y);}

\foreach \x/\y in 
{1/2,4/3,6/5,8/7}
{\draw[decorate,decoration={snake, amplitude=1pt, segment length=8pt}]
(a\x) -- (a\y);}

\foreach \x/\y in 
{5/1,6/2,7/3,8/4}
{\path[edge,dashed] (a\x)--(a\y);}
\end{scope}

\end{scope}

\begin{scope}[shift={(6,6.5)}]
\foreach \x/\y/\z in
{1.5/1.5/1,1.5/-1.5/2,-1.5/-1.5/3,-1.5/1.5/4,3/3/5,3/-3/6,-3/-3/7,-3/3/8,4.5/4.5/9,4.5/-4.5/10,-4.5/-4.5/11,-4.5/4.5/12,6/6/13,6/-6/14,-6/-6/15,-6/6/16}
{\node[verti] (a\z) at (\x,\y){};}

\foreach \x/\y/\z in
{1.1/2/11,1.1/-2/5,-1.1/-2/12,-1.1/2/1,2.6/3.5/3,2.6/-3.5/9,-2.6/-3.5/7,-2.6/3.5/14,4.1/5/16,4.1/-5/4,-4.1/-5/15,-4.1/5/8,5.6/6.5/2,5.6/-6.5/10,-5.6/-6.5/6,-5.6/6.5/13}
{\node[ver] () at (\x,\y){\tiny{$T_{\z}^4$}};}

\foreach \x/\y in 
{1/4,2/3,6/7,8/5,10/11,12/9,14/15,16/13}
{\path[edge] (a\x)--(a\y);}

\foreach \x/\y in 
{5/9,6/10,7/11,8/12}
{\path[edge,dashed] (a\x)--(a\y);}

\foreach \x/\y in 
{1/4,2/3,6/7,8/5,10/11,12/9,14/15,16/13}
{\draw [line width=2pt, line cap=round, dash pattern=on 0pt off 1.3\pgflinewidth]  (a\x) -- (a\y);}

\foreach \x/\y in 
{1/2,4/3,6/5,8/7,10/9,12/11,14/13,16/15}
{\draw[decorate,decoration={snake, amplitude=1pt, segment length=8pt}]
(a\x) -- (a\y);}

\foreach \x/\y in 
{5/1,6/2,7/3,8/4,9/13,10/14,11/15,12/16}
{\path[edge] (a\x)--(a\y);}

\draw[edge,dashed] plot [smooth,tension=0.5] coordinates{(a1) (4.2,2.5) (a13)};

\draw[edge,dashed] plot [smooth,tension=0.5] coordinates{(a2) (4.2,-2.5) (a14)};

\draw[edge,dashed] plot [smooth,tension=0.5] coordinates{(a3) (-4.2,-2.5) (a15)};

\draw[edge,dashed] plot [smooth,tension=0.5] coordinates{(a4) (-4.2,2.5) (a16)};

\end{scope}

\begin{scope}[shift={(-4,11)}]
\foreach \x/\y/\z/\w in 
{-1/1/T_i^1/1,2/1/T_i^2/2,2/-1.5/T_i^4/4,-1/-1.5/T_i^3/3}
{\node[ver] (\w) at (\x,\y){\tiny{$\z$}};}

\path[edge,dashed] (1)--(2);
\path[edge] (2)--(3);
\path[edge,dotted] (3)--(4);

\end{scope}
\node[ver] () at (-12,-0.4){$(a)$}; 
\end{scope}

\begin{scope}[shift={(0,15.5)}]
\begin{scope}[shift={(-13,16)}]
\foreach \x/\y/\z in
{1.5/1.5/1,1.5/-1.5/2,-1.5/-1.5/3,-1.5/1.5/4,3/3/5,3/-3/6,-3/-3/7,-3/3/8,4.5/4.5/9,4.5/-4.5/10,-6/-4.5/11,-6/4.5/12,6/6/13,6/-6/14,-7.5/-6/15,-7.5/6/16,-9/-6/17,-9/6/18,-4.5/-3/19,-4.5/3/20}
{\node[verti] (a\z) at (\x,\y){};}

\foreach \x/\y/\z in
{1.1/2/T_{12}^1,1.1/-2/T_4^1,-1.1/-2/T_{8}^2,-1.1/2/T_{11}^2,2.6/3.5/T_{7}^1,2.6/-3.5/T_{10}^1,-2.6/-3.5/T_{13}^2,-2.6/3.5/T_{3}^2,4.1/5/T_{15}^1,4.1/-5/T_{5}^1,-5.6/-5/T_{11}^1,-5.6/5/T_{8}^1,5.6/6.5/T_{6}^1,5.6/-6.5/T_{9}^1,-7.1/-6.5/T_{3}^1,-7.1/6.5/T_{13}^1,-9.7/-6/T_{2}^4,-9.7/6/T_1^3,-3.8/-2.8/T_2^3,-3.8/2.8/T_{1}^4}
{\node[ver] () at (\x,\y){\tiny{$\z$}};}

\foreach \x/\y in 
{5/9,6/10,19/11,20/12,1/4,2/3,6/7,8/5,10/11,12/9,14/15,16/13,15/17,16/18}
{\path[edge] (a\x)--(a\y);}

\foreach \x/\y in 
{1/4,2/3,6/7,8/5,10/11,12/9,14/15,16/13}
{\draw [line width=2pt, line cap=round, dash pattern=on 0pt off 1.3\pgflinewidth]  (a\x) -- (a\y);}

\foreach \x/\y in 
{1/2,4/3,6/5,8/7,10/9,12/11,14/13,16/15}
{\draw[decorate,decoration={snake, amplitude=1pt, segment length=8pt}]
(a\x) -- (a\y);}

\foreach \x/\y in 
{5/1,6/2,7/3,8/4,9/13,10/14,11/15,12/16}
{\path[edge,dotted] (a\x)--(a\y);}

\draw[edge] plot [smooth,tension=0.5] coordinates{(a1) (4.2,2.5) (a13)};

\draw[edge] plot [smooth,tension=0.5] coordinates{(a2) (4.2,-2.5) (a14)};

\end{scope}

\begin{scope}[shift={(5.5,18.5)}]

\foreach \x/\y/\z in
{1.5/1.5/1,1.5/-1.5/2,-1.5/-1.5/3,-1.5/1.5/4,3/3/5,3/-3/6,-3/-3/7,-3/3/8}
{\node[verti] (a\z) at (\x,\y){};}

\foreach \x/\y/\z in
{1.1/2/T_{15}^2,1.1/-2/T_{5}^2,-1.1/-2/T_{12}^2,-1.1/2/T_{4}^2,2.6/3.5/T_{6}^4,2.6/-3.5/T_{9}^4,-2.6/-3.5/T_{7}^4,-2.6/3.5/T_{10}^4}
{\node[ver] () at (\x,\y){\tiny{$\z$}};}

\foreach \x/\y in 
{1/4,2/3,6/7,8/5}
{\path[edge] (a\x)--(a\y);}

\foreach \x/\y in 
{1/4,2/3,6/7,8/5}
{\draw [line width=2pt, line cap=round, dash pattern=on 0pt off 1.3\pgflinewidth]  (a\x) -- (a\y);}

\foreach \x/\y in 
{1/2,4/3}
{\draw[decorate,decoration={snake, amplitude=1pt, segment length=8pt}]
(a\x) -- (a\y);}

\foreach \x/\y in 
{5/1,6/2,7/3,8/4}
{\path[edge,dotted] (a\x)--(a\y);}
\end{scope}

\begin{scope}[shift={(-14,4.5)}]

\foreach \x/\y/\z in
{1.5/1.5/1,1.5/-1.5/2,-1.5/-1.5/3,-1.5/1.5/4,3/3/5,3/-3/6,-3/-3/7,-3/3/8}
{\node[verti] (a\z) at (\x,\y){};}

\foreach \x/\y/\z in
{1.1/2/T_{15}^3,1.1/-2/T_{5}^3,-1.1/-2/T_{12}^3,-1.1/2/T_{4}^3,2.6/3.5/T_{7}^1,2.6/-3.5/T_{10}^1,-2.6/-3.5/T_{6}^1,-2.6/3.5/T_{9}^1}
{\node[ver] () at (\x,\y){\tiny{$\z$}};}

\foreach \x/\y in 
{1/4,2/3}
{\path[edge] (a\x)--(a\y);}

\foreach \x/\y in 
{1/4,2/3}
{\draw [line width=2pt, line cap=round, dash pattern=on 0pt off 1.3\pgflinewidth]  (a\x) -- (a\y);}

\foreach \x/\y in 
{1/2,4/3,6/5,8/7}
{\draw[decorate,decoration={snake, amplitude=1pt, segment length=8pt}]
(a\x) -- (a\y);}

\foreach \x/\y in 
{5/1,6/2,7/3,8/4}
{\path[edge,dashed] (a\x)--(a\y);}

\end{scope}

\begin{scope}[shift={(6,5)}]
\foreach \x/\y/\z in
{1.5/1.5/1,1.5/-1.5/2,-1.5/-1.5/3,-1.5/1.5/4,3/4.5/5,3/-3/6,-3/-3/7,-3/4.5/8,4.5/6/9,4.5/-4.5/10,-4.5/-4.5/11,-4.5/6/12,6/7.5/13,6/-6/14,-6/-6/15,-6/7.5/16,1.5/3/17,-1.5/3/18,6/9/19,-6/9/20}
{\node[verti] (a\z) at (\x,\y){};}

\foreach \x/\y/\z in
{1/2/T_{11}^4,1.1/-2/T_{5}^4,-1.1/-2/T_{12}^4,-1.1/2/T_{1}^4,2.6/5/T_{13}^3,2.6/-3.5/T_{9}^4,-2.6/-3.5/T_{7}^4,-2.6/5/T_{2}^3,4.1/6.5/T_{1}^3,4.1/-5/T_{4}^4,-4.1/-5/T_{15}^4,-4.1/6.5/T_{11}^3,5.5/8/T_{2}^4,5.6/-6.5/T_{10}^4,-5.6/-6.5/T_{6}^4,-5.4/8/T_{13}^4,1.1/3.5/T_{3}^2,-1.1/3.5/T_{8}^1,5.6/9.5/T_{3}^1,-5.6/9.5/T_{8}^2}
{\node[ver] () at (\x,\y){\tiny{$\z$}};}

\foreach \x/\y in 
{1/4,2/3,6/7,8/5,10/11,12/9,14/15,16/13}
{\path[edge] (a\x)--(a\y);}

\foreach \x/\y in 
{5/9,6/10,7/11,8/12}
{\path[edge,dashed] (a\x)--(a\y);}

\foreach \x/\y in 
{1/4,2/3,6/7,8/5,10/11,12/9,14/15,16/13}
{\draw [line width=2pt, line cap=round, dash pattern=on 0pt off 1.3\pgflinewidth]  (a\x) -- (a\y);}

\foreach \x/\y in 
{1/2,4/3,6/5,8/7,10/9,12/11,14/13,16/15}
{\draw[decorate,decoration={snake, amplitude=1pt, segment length=8pt}]
(a\x) -- (a\y);}

\foreach \x/\y in 
{17/1,6/2,7/3,18/4,19/13,10/14,11/15,20/16}
{\path[edge] (a\x)--(a\y);}

\draw[edge,dashed] plot [smooth,tension=0.5] coordinates{(a1) (4.2,2.5) (a13)};

\draw[edge,dashed] plot [smooth,tension=0.5] coordinates{(a2) (4.2,-2.5) (a14)};

\draw[edge,dashed] plot [smooth,tension=0.5] coordinates{(a3) (-4.2,-2.5) (a15)};

\draw[edge,dashed] plot [smooth,tension=0.5] coordinates{(a4) (-4.2,2.5) (a16)};

\end{scope}

\begin{scope}[shift={(-5,7)}]
\foreach \x/\y/\z/\w in 
{-1/1/T_i^1/1,2/1/T_i^2/2,2/-1.5/T_i^4/4,-1/-1.5/T_i^3/3}
{\node[ver] (\w) at (\x,\y){\tiny{$\z$}};}

\path[edge,dashed] (1)--(2);
\path[edge] (2)--(3);
\path[edge,dotted] (3)--(4);

\end{scope}
\node[ver] () at (-12,-0.6){$(b)$ }; 
\end{scope}

\begin{scope} [shift = {(-15,13)}]
\foreach \x/\y/\z in {1/-0.5/0,5/-0.5/1,9/-0.5/2,13/-0.5/3,17/-0.5/4}
{\node[ver] () at (\x,\y){$\z$};}
\path[edge,dashed] (12,0) -- (14,0);
\path[edge] (16,0) -- (18,0);
\draw [line width=2pt, line cap=round, dash pattern=on 0pt off 1.3\pgflinewidth]  (16,0) -- (18,0);
\path[edge] (8,0) -- (10,0);
\draw[decorate,decoration={snake, amplitude=1pt, segment length=8pt}] (0,0) -- (2,0);

\path[edge,dotted] (4,0) -- (6,0);
\end{scope}  

\end{tikzpicture}

\caption{$(a)$ The gem $\Gamma^2$ of $\mathbb{RP}^3\times \mathbb{S}^1$ with $48$ vertices, $(b)$ The crystallization $\Gamma^1$ of $\mathbb{RP}^3\times \mathbb{S}^1$ with $40$ vertices.}\label{fig2}

\end{figure}

\begin{lemma}\label{lemma:D-J} There are $1+2^{n-1}$ small covers over $\Delta^{n-1}\times I$ up to D-J equivalence, where $n\ge 3$.
\end{lemma}

\begin{proof}
Let us consider the vertex $v_0$. The vertex $v_0$ belongs to precisely $n$ codimension-one faces $F_1,F_2,\dots,F_n$. Let us denote the $\mathbb Z_2$-characteristic vector of $F_j$ by $b_j$ for all $0\le j\le n+1$. By the definition of $\mathbb Z_2$-characteristic function, $\{b_1,b_2,\dots,b_n\}$ is a basis of $\Z$. Now, if we consider the vertex $v_k$, since it lies in exactly $n$ codimension-one faces $F_0,F_1,\dots,F_{k-1},F_{k+1},\dots,F_n$, the set of $\mathbb Z_2$-characteristic vectors $\{b_0,b_1,\dots,b_{k-1},\hat{b}_k,b_{k+1},\dots,b_n\}$ also forms a basis of $\Z$ for all $1\le k \le n-1$. This implies $b_0$ can either be $b_1+b_2+\dots+b_{n-1}$ or $b_1+b_2+\dots+b_{n}$. If $b_0=b_1+b_2+\dots+b_{n-1}$, then $b_{n+1}$ can be any vector that is not spanned by $\{b_j:0\le j \le n-1\}$. Thus, in this case, $b_{n+1}$ has $2^{n-1}$ choices. Now, if $b_0=b_1+b_2+\dots+b_{n}$, then $\{b_j\ |\ 0\le j\le n-1\}$ is a basis of $\Z$. Since $\{b_0,\dots,b_{k-1}, \hat{b}_k, b_{k+1},\dots b_{n-1},b_{n+1}\}$ forms a basis of $\Z$ for all $0\le k\le n-1$, we get that $b_{n+1}=\sum_{j=0}^{n-1} b_j$, i.e., $b_{n+1}=b_n$. Therefore, a $\mathbb Z_2$-characteristic function on $\mathcal F(\Delta^{n-1}\times I)$ is obtained by the following two methods. 

\smallskip

\noindent \textbf{First Method:} The first method is to assign $n-1$ linearly independent vectors to any $n-1$ faces of the set $\{F_j\ |\ 0\le j\le n-1\}$ and assign the sum of these $n-1$ linearly independent vectors to the remaining face of this set, then each of $F_n$ and $F_{n+1}$ can be assigned any vector that is not in the span of these $n-1$ linearly independent vectors. 

\smallskip

\noindent \textbf{Second Method:} The second method is to assign $n$ linearly independent vectors to $\{F_j\ |\ 0\le j\le n-1\}$, then as proved above, $b_n=b_{n+1}=\sum_{j=0}^{n-1} b_j$.

\smallskip

\noindent By simple calculation, one gets that the number of different $\mathbb Z_2$-characteristic functions that are obtained via the first method is $2^{2n-2}(2^n-1)(2^n-2)\dots(2^n-2^{n-2})$ and that are obtained using the second method is $(2^n-1)(2^n-2)\dots(2^n-2^{n-1})$. Therefore, corresponding to a fixed basis $\beta$ of $\Z$, we have exactly one choice by the second method and $2^{n-1}$ choices by the first method for $\lambda$. Since two small covers $M^n(\lambda_1)$ and $M^n(\lambda_2)$ are D-J equivalent if and only if there exists an automorphism $\theta$ of $\Z$ such that $\lambda_2=\theta \circ \lambda_1$, we conclude that there are $1+2^{n-1}$ D-J equivalence classes of small covers over $\Delta^{n-1}\times I$, where $n\ge 3$.
\end{proof}

 Let $\Delta^n_1=[v_0^1,v_1^1,\cdots,v_n^1]$ and $\Delta^n_2=[v_0^2,v_1^2,\cdots,v_n^2]$ be two $n$-simplices. For each $j\in \{1,2\}$, let $\lambda_j:\mathcal F(\Delta^n_j)\to \Z$ be a $\mathbb{Z}_2$-characteristic function defined by $\lambda_j([v_0^j,\cdots,v_{i-1}^j,v_{i+1}^j,\cdots,v_n^j])=b_i$, for all $1\le i\le n$, where $\{b_1,b_2,\cdots,b_n\}$ is a basis of $\Z$ (cf. Subsection \ref{3.1}). Fix an index $i\in \{0,1,\dots,n\}$. Let  $\Delta^n_1 \# \Delta^n_2$ be formed by removing a small ball around the vertex $v_i^j$ from $\Delta^n_j$, for each $j\in\{1,2\}$, and gluing the results together. Then $\Delta^n_1 \# \Delta^n_2$ is homeomorphic to $\Delta^{n-1}\times I$. Note that \(\Delta^{n-1} \times I\) (\(\cong \Delta^n_1 \# \Delta^n_2\)) inherits a \(\mathbb{Z}_2\)-characteristic function that satisfies the conditions of the second method, and conversely, a \(\mathbb{Z}_2\)-characteristic function on \(\Delta^{n-1} \times I\) that satisfies the conditions of the second method corresponds to one on \(\Delta^n_1 \# \Delta^n_2\). Therefore, the small cover $M^n(\lambda)$ over $\Delta^{n-1}\times I$ is $\mathbb{RP}^n \# \mathbb{RP}^n$ whenever $\lambda$ is a $\mathbb{Z}_2$-characteristic function satisfying the conditions of the second method. For details, see \cite[1.11]{dj91}. Now, we will focus on studying the small covers over $\Delta^{n-1}\times I$ when $\lambda$ satisfies the conditions of the first method.

\begin{lemma}\label{on}
Let $\lambda:\mathcal F(\Delta^{n-1}\times I)\to \Z$ be a $\mathbb Z_2$-characteristic function satisfying the conditions of the first method, where $n\ge 3$. If $n$ is odd, then the small cover $M^n(\lambda)$ over $\Delta^{n-1}\times I$ is non-orientable. When $n$ is even, $M^n(\lambda)$ is orientable if $b_{n+1}$ is the sum of $b_n$ and an even number of $b_i$, and $M^n(\lambda)$ is non-orientable if $b_{n+1}$ is the sum of $b_n$ and an odd number of $b_i$, where $1\le i \le n-1$.
\end{lemma} 
\begin{proof}
We know that a gem $\Gamma$ of a closed manifold $M$ is bipartite if and only if $M$ is orientable. If \(\Gamma\) is bipartite, we designate the vertices of one set as positive vertices and the vertices of the other set as negative vertices. Let $\lambda:\mathcal F(\Delta^{n-1}\times I)\to \Z$ be a $\mathbb Z_2$-characteristic function satisfying the conditions of the first method, and let $(\Gamma,\gamma)$ be the gem of $M^n(\lambda)$ that is obtained using the above construction. 
\begin{equation}\nonumber
(\ast) = \begin{cases}
\begin{array}{l}
\text{If we remove all the $n$-colored edges incident to $T_i^1$ and all the $0$-colored edges} \\
\text{incident to $T_i^n$, where $1 \le i \le 2^n$, then we get two connected components of $\Gamma$} \\
\text{with $n2^{n-1}$ vertices in each component.}
\end{array}
\end{cases}
\end{equation}

Let us denote these components by $G_1$ and $G_2$. Since we are not considering $\mathbb{Z}_2$-characteristic vectors corresponding to $(n-1)$-faces $\Delta^{n-1}\times \partial(I)$ of $\Delta^{n-1}\times I$, the components $G_1$ and $G_2$ both represent $\mathbb{RP}^{n-1}\times I$, and hence $\Gamma$ represents an $\mathbb{RP}^{n-1}$-bundle over $\mathbb{S}^{1}$ (cf. Proposition \ref{Proposition: product} and Figure \ref{table}).
Let $n$ be an odd number. By the fact $(\ast)$, we have that $G_1$ and $G_2$ both represent $\mathbb{RP}^{n-1}\times I$, which is non-orientable. Therefore, $M^n(\lambda)$ is non-orientable.

Now, let us assume that $n$ is an even number. If we connect $G_1$ and $G_2$ by all the $n$-colored edges of $\Gamma$ that are incident to $T_i^1$, where $1\le i\le 2^n$, then we get $G_3$ that represents $\mathbb{RP}^{n-1}\times I$, which is orientable. Thus, $G_3$ is bipartite. Now, $T_k^n$ is joined with $T_l^n$ by an edge of color $0$ if and only if $g_k+b_{n+1}=g_l$, i.e., $g_k+b_n+\sum_{j=1}^{n-1}c_j b_j=g_l$, where $c_j=1$ or $0$. If the cardinality of the set $\{j\ |\ 1\le j\le n-1, c_j=1\}$ is odd, then $T_k^n$ and $T_l^n$ both are either positive or negative vertices. This implies that $M^n(\lambda)$ is non-orientable. Now, if the cardinality of the set $\{j\ |\ 1\le j\le n-1, c_j=1\}$ is even, then one of $T_k^n$ and $T_l^n$ is a positive vertex, and the other is a negative vertex, which implies that $M^n(\lambda)$ is orientable.
\end{proof}

\begin{remark}\label{remark: count}
{\rm There are $2^{2n-3}(2^n-1)(2^n-2)\dots(2^n-2^{n-2})$ different $\mathbb Z_2$-characteristic functions (obtained using the first method) that induce orientable small covers, and the same number of $\mathbb Z_2$-characteristic functions (obtained using the first method) induce non-orientable small covers. In conclusion, we get $2^{n-2}$ orientable (resp. non-orientable) small covers up to D-J equivalence, when $n$ is even. 
}
 \end{remark}

 Since the $\mathbb Z_2$-characteristic function in Example \ref{ex} satisfies the conditions of the first method, the component $G_1$ consists of vertices $T_1^j, T_{14}^j, T_{8}^j, T_{13}^j, T_{16}^j, T_{2}^j, T_{11}^j,$ and $T_{3}^j$; the other component $G_2$ consists of $T_6^j, T_{15}^j, T_{7}^j, T_{12}^j, T_{4}^j, T_{10}^j,$ $ T_{5}^j,$ and $T_{9}^j$, where $1\le j\le 4$. Since $n=4$ and $b_{n+1}=b_n$, the small cover $M^4(\lambda)$ is an orientable manifold.

\begin{theorem}\label{main}
    Let $\lambda$ be a $\mathbb Z_2$-characteristic function of $\Delta^{n-1} \times I$ obtained by the first method, where $n\ge 3$. Then the corresponding small cover $M^n(\lambda)$ admits a $2^{n-1}(n+1)$-vertex crystallization $(\Gamma^{\prime},\gamma^{\prime})$. Furthermore, the regular genus of $\Gamma^{\prime}$ is $2$ for $n = 3$, and $1 + 2^{n-4}(n^2 - 2n - 3)$ for $n \geq 4$.
\end{theorem}

\begin{proof}
Let us consider the gem $\Gamma$ of $M^n(\lambda)$ that is obtained using the construction given above (for an example, see Figure \ref{fig1}(a)). Due to the fact $(\ast)$ as in the proof of Lemma \ref{on}, we have $G_1$ and $G_2$ both representing $\mathbb{RP}^{n-1}\times I$. Since elements of $\Z$ are involutory, we get from the generalization of Figure \ref{table} that $g_{\{0,1\}}=g_{\{n-1,n\}}=2^{n-2}(n-1)$, of which $2^{n-2}$ are eight-cycles (each eight-cycle involves vertices of both $G_1$ and $G_2$) and $2^{n-2}(n-2)$ are four-cycles. Also, $g_{\{i,i+1\}}=2^{n-2}(n-1)$, of which $2^{n-1}$ are six-cycles ($G_1$ and $G_2$ each have $2^{n-2}$ six-cycles) and $2^{n-2}(n-3)$ are four-cycles, for all $1\le i\le n-2$. For all other pairs of colors, we have $g_{\{k,l\}}=n2^{n-2}$, where all the components are four-cycles. It is clear that $g_{\hat{0}}=g_{\hat{n}}=1$ and $g_{\hat{j}}=2$ for all $1\le j\le n-1$. To obtain a crystallization from this gem of $M^n(\lambda)$, we will use polyhedral glue moves.

Consider color \(n-1\). Let \(\Lambda_1\) and \(\Lambda_1'\) be two \((n-2)\)-cubes in \(G_1\) regularly colored by \(\{0,1,\dots,n-3\}\), consisting of the vertices \(T_i^1\) and \(T_i^2\), respectively, where \(1 \le i \le 2^n\). There exists an isomorphism \(\Phi_1\) between these two cubes such that \(T_k^1\) and \(\Phi(T_k^1)\) are connected by an edge colored \(n-1\). From the generalization of Figure \ref{table}, we always have such \(\Lambda_1\) and \(\Lambda_1'\). Therefore, we apply a polyhedral glue move with respect to \((\Phi_1, \Lambda_1, \Lambda_1', n-1)\) to get another gem \(\Gamma^1\) of \(M^n(\Lambda)\) with \(n2^n - 2^{n-1}\) vertices. In the colored triangulation of \(M^n(\Lambda)\), \(\mathcal{K}(\Gamma^1)\), there is exactly 1 vertex colored \(n-1\), 1 vertex colored 0, 1 vertex colored \(n\), and 2 vertices colored \(j\) for all \(1 \le j \le n-2\).

As a consequence of this polyhedral move, $2^{n-2}$ eight-cycles colored by $\{n-1,n\}$ become six-cycles, $2^{n-2}$ six-cycles colored by $\{n-2,n-1\}$ in $G_1$ become $4$-cycles, and $2^{n-3}$ four-cycles colored by $\{k,n-1\}$ gets removed completely for all $k\in \Delta_n\backslash \{n-2,n-1,n\}$. Each pair of four-cycles, out of $2^{n-2}$ pairs, colored by $\{n-2,n\}$ merges to form $2^{n-2}$ six-cycles, and $2^{n-3}$ four-cycles colored by $\{k,l\}$ are completely removed when both $k, l\in \{0,1,\dots,n-3\}$. Each pair of four-cycles, out of $2^{n-3}$ pairs, colored by $\{k,l\}$ merges to form $2^{n-3}$ four-cycles when either $k$ or $l$, but not both, belongs to $\{0,1,\dots,n-3\}$. In brief, all the eight-cycles colored by $\{n-1,n\}$ in $\Gamma$ are six-cycles in $\Gamma^1$. There are $2^{n-2}$ six-cycles colored by $\{n-2,n\}$ in $\Gamma^1$, and all the other bi-colored cycles are either four-cycles or remain the same cycles as in $\Gamma$. 

Apply a polyhedral move with respect to $(\Phi_j, \Lambda_{j}, \Lambda_{j}^\prime, n-j)$ to obtain $\Gamma^j$ for all $2\le j\le n-1$ successively, where $\Lambda_j$ and $\Lambda_j^\prime$ are two $(n-2)$-cubes in $G_1$ when $j$ is odd (resp. in $G_2$ when $j$ is even), regularly colored by $\{0,1,\dots,n-j-2,n-j+2,n-j+3,\dots,n\}$, consisting of the vertices $T_i^j$ and $T_i^{j+1}$, respectively, where $1\le i \le 2^n$. Due to the fact $(\ast)$, a polyhedral move with respect to $(\Phi_k, \Lambda_{k}, \Lambda_{k}^\prime, n-k)$ does not change the size of a bi-colored cycle that was introduced in $\Gamma^j$ for $1\le j\le k-1$. In $\Gamma^k$, $2^{n-2}$ six-cycles colored by $\{n-k-1,n-k+1\}$ are introduced, where each six-cycle is formed as a result of the merger of two four-cycles because neither $k-1$ nor $k-2$ is in $\{0,1,\dots,n-k-2,n-k,n-k+2,n-k+3,\dots,n\}$. If $k=n-1$, then all the eight-cycles colored by $\{0,1\}$ become six-cycles in $\Gamma^{n-1}$. 

Finally, after applying all the polyhedral moves as explained above, we get a crystallization $\Gamma^\prime=\Gamma^{n-1}$. In $\Gamma^\prime$, we have $2^{n-2}$ six-cycles for each of the color pairs in $S= \{\{n-1,n\},$ $\{n-2,n\},\{n-3,n-1\},\dots,\{1,3\},\{0,2\},\{0,1\} \}$, and all the other bi-colored cycles are of length $4$. Also, note that the number of vertices in $\Gamma^\prime$ is $n2^n-(n-1)2^{n-1}=(n+1)2^{n-1}$. Therefore, we have $g_{\{k,l\}}=(n+1)2^{n-3}$ when $\{k,l\}$ does not belong to $S$, and $g_{\{k,l\}}=n2^{n-3}$ when $\{k,l\}\in S$.

For $n=3$, we have $S=\{\{2,3\}, \{1,3\}, \{0,2\}, \{0,1\}\}$. Therefore, the regular genus $\rho(\Gamma^{\prime})=\rho_\varepsilon(\Gamma^{\prime})=2$, where $\varepsilon=(0,3,1,2)$. If $n> 3$, we calculate the regular genus of $\Gamma^\prime$ with respect to the permutation $\varepsilon=(\varepsilon_0,\varepsilon_1,\dots,\varepsilon_n)=(0,n,1,2,\dots,n-1)$,
$$\rho_\varepsilon(\Gamma^\prime)=1-\frac{1}{2}\left(\sum_{i \in \mathbb{Z}_{n+1}}g_{\varepsilon_i\varepsilon_{i+1}} + (1-n)\frac{V(\Gamma^\prime)}{2}\right)=1-\frac{1}{2}\left( (n+1)^2 2^{n-3}+ (1-n) (n+1)2^{n-2}\right)$$ 
$$=1-\frac{1}{2}\left( (n+1) 2^{n-3} ((n+1)+ (1-n)2)\right)=1-\frac{1}{2}\left( (n+1) 2^{n-3} (3-n)\right)=1+2^{n-4}(n^2-2n-3).$$

\noindent The permutation $\varepsilon=(\varepsilon_0,\varepsilon_1,\dots,\varepsilon_n)=(0,n,1,2,\dots,n-1)$ maximizes the sum of the number of bi-colored cycles, since no pair $(\varepsilon_i,\varepsilon_{i+1})$ belongs to $S$, for $i\in \mathbb{Z}_{n+1}$. Hence, $\rho_\varepsilon(\Gamma^\prime)$ is the  regular genus of the crystallization $\Gamma^\prime$ of $M^n(\lambda)$. 
\end{proof}

In Example \ref{ex}, we first apply the polyhedral glue move with respect to $(\Phi_1,\Lambda_1,\Lambda_1^{\prime},3)$, where $\Lambda_1=\{T_1^1,T_{16}^1,T_{2}^1,T_{14}^1\}$, $\Lambda_1^{\prime}=\{T_1^2,T_{16}^2,T_{2}^2,T_{14}^2\}$ and $\Phi_1(T_k^1)=T_k^2$ for all $T_k^1 \in \Lambda_1$ (Figures \ref{fig1}(a), \ref{fig1}(b)). Then, we apply the polyhedral glue move with respect to $(\Phi_2,\Lambda_2,\Lambda_2^{\prime},2)$, where $\Lambda_2=\{T_{10}^2, T_{6}^2, T_{9}^2, T_{7}^2\}$, $\Lambda_2^{\prime}=\{T_{10}^3,T_{6}^3,T_{9}^3,T_{7}^3\}$ and $\Phi_2(T_k^2)=T_k^3$ for all $T_k^2 \in \Lambda_2$ (Figure \ref{fig2}(a)). Finally, applying the polyhedral glue move with respect to $(\Phi_3,\Lambda_3,\Lambda_3^{\prime},1)$, where $\Lambda_3=\{T_{16}^3, T_{8}^3, T_{14}^3, T_{3}^3\}$, $\Lambda_2^{\prime}=\{T_{16}^4, T_{8}^4, T_{14}^4, T_{3}^4\}$ and $\Phi_3(T_k^3)=T_k^4$ for all $T_k^3 \in \Lambda_3$, we get a crystallization of $M^4(\lambda)=\mathbb{RP}^3\times\mathbb{S}^1$ with $40$ vertices (Figure \ref{fig2}(b)).

\begin{remark}
{\rm If $\lambda:\mathcal F(\Delta^{n-1} \times I) \to \Z $ is a $\mathbb{Z}_2$-characteristic function satisfying the conditions of the second method, then the same polyhedral glue moves as in the proof of Theorem \ref{main} will give a $2^{n-1}(n+1)$-vertex crystallization of $\mathbb{RP}^n \# \mathbb{RP}^n$ with regular genus $1+2^{n-4}(n^2-2n-3).$ 

However, a $(2^{n+1} - 2)$-vertex crystallization of $\mathbb{RP}^n \# \mathbb{RP}^n$ with regular genus $2 + 2^{n-2}(n - 3)$ is known:
Take two copies of the $2^n$-vertex crystallization $\Gamma$ of $\mathbb{RP}^n$. If we delete one vertex from each copy and join the hanging edges of the same colors, then the resulting graph $\Gamma \# \Gamma$, called the graph connected sum, is a crystallization of $\mathbb{RP}^n \# \mathbb{RP}^n$ with $(2^{n+1} - 2)$ vertices. Since each $g_{\{i,j\}} = 4$ in the $2^n$-vertex crystallization $\Gamma$ of $\mathbb{RP}^n$, the crystallization $\Gamma \# \Gamma$ has $2^{n-1} - 1$ components for any pair of colors. Therefore, the regular genus of $\Gamma \# \Gamma$ is $2 + 2^{n-2}(n - 3)$.
}
\end{remark}

Consider \( n = 4 \) in Theorem \ref{main}. Since we know that, corresponding to a \(\mathbb{Z}_2\)-characteristic function \(\lambda:\Delta^3 \times I \to \mathbb{Z}_2^4\) that satisfies the conditions of the first method, the small cover \( M^4(\lambda) \) is an \(\mathbb{RP}^3\)-bundle over \(\mathbb{S}^1\), we have that the Euler characteristic of \( M^4(\lambda) \) is \(\chi(\mathbb{RP}^3) \cdot \chi(\mathbb{S}^1) = 0\) and \(\mathrm{rk}(\pi_1(M^4(\lambda))) = 2\). Therefore, from Proposition \ref{lbrg}, it follows that the regular genus of \( M^4(\lambda) \) is greater than or equal to $6$. Hence, Theorem \ref{main} and Remark \ref{remark: count} implies the following result.

\begin{corollary} \label{corollary:regular genus 6}
There are $4$ orientable and $4$ non-orientable small covers with regular genus $6$ over $\Delta^3\times I$ up to D-J equivalence, corresponding to $\mathbb{Z}_2$-characteristic functions satisfying the conditions of the first method. All these $8$ small covers over $\Delta^3\times I$ represent some $\mathbb{RP}^3$-bundle over $\mathbb{S}^1.$ 
\end{corollary}

\begin{remark}\label{remark: PL isomorphic}
    {\rm These four orientable (resp.\ non-orientable) small covers over $\Delta^3 \times I$  are not D-J equivalent but are PL homeomorphic, since the crystallizations, obtained by applying Theorem \ref{main}, of these four orientable (resp.\ non-orientable) small covers are isomorphic, with the following isomorphism signature (obtained using Regina \cite{burton2013regina}): 

    \noindent OvvvvLLAAwvwzQQwPAQMMQQQwzMMQQQQQQclimpjnruywzBtBzyzDxyxCxFFDGE\\vuBxCDFCvIJKLKLLKMIINJNMKMNJIMJNLaaaaaaaaaaaaaaaaaaaaaaaaaaaaaaaaaaaaaa\\aaaaaaaaaaaaaaaaaaaaaaaaaaaaaaaaaaaaaaaaaaaaaaaaaaaaaaaaaaaaaaaaaaaaaaaaaaaaaaaaa\\aaa

    \noindent (resp. OvvvvLLAAwvwzQQwPAQMMwQQMPQQQQQLQQclimpjnruywzBtBzyzDxyxCxF\\FDGEBxCDFCIEKHJEGHCLJILJJLGKKNMNMMNNMaaaaaaaaaaaaaaaaaaaaaaaaaaaaa\\aaaaaaaaaaaaaaaaaaaaaaaaaaaaaaaaaaaaaaaaaaaaaaaaaaaaaaaaaaaaaaaaaaaaaaaaaaaaaaaa\\aaaaaaaaaaaaa).

    }
\end{remark}

 In \cite{bb18}, the notion of weak semi-simple crystallization was introduced. Let $M$ be a closed, connected 4-manifold, and let $m$ denote the rank of the fundamental group of $M$. A crystallization $(\Gamma,\gamma)$ of $M$ with color set $\Delta_4$ is called a weak semi-simple crystallization if there exists a permutation $(\varepsilon_0, \varepsilon_1, \dots, \varepsilon_4)$ of $\Delta_4$ such that $g_{\{\varepsilon_i,\varepsilon_{i+1},\varepsilon_{i+2}\}} = m + 1$ for $i \in \Delta_4$ (with the subscript of $\varepsilon$ taken modulo 5). It was shown that the regular genus of $M$ attains the lower bound of Proposition \ref{lbrg} if and only if $M$ admits a weak semi-simple crystallization. Since the \(\mathbb{RP}^3\)-bundles over \(\mathbb{S}^1\) described in Corollary \ref{corollary:regular genus 6} satisfy the equality \(\mathcal{G}(M) = 2\chi(M) + 5m - 4\), the crystallizations constructed here, with regular genus $6$, are weak semi-simple crystallizations.

\bigskip

\noindent {\bf Acknowledgment:}  The authors would like to thank the anonymous referees for many useful comments and suggestions. The first author is supported by the Institute fellowship from the Indian Institute of Technology Delhi. The second author is supported by the Mathematical Research Impact Centric Support (MATRICS) Research Grant (MTR/2022/000036) from ANRF (India).

\end{document}